\documentclass{article}
\usepackage[utf8]{inputenc}
\usepackage{amsmath,amssymb,amsthm,amsfonts,latexsym,graphicx,subfigure,mathtools}

\usepackage{url}
\usepackage{xcolor}

\DeclareMathOperator{\inp}{in}
\DeclareMathOperator{\out}{out}
\DeclareMathOperator{\id}{Id}

\theoremstyle{plain}
\newtheorem{theorem}{Theorem}[section]
\newtheorem{lemma}[theorem]{Lemma}
\newtheorem{proposition}[theorem]{Proposition}
\newtheorem{corollary}[theorem]{Corollary}
\newtheorem{conjecture}[theorem]{Conjecture}

\theoremstyle{definition}
\newtheorem{definition}[theorem]{Definition}
\newtheorem{example}[theorem]{Example}
\theoremstyle{remark}
\newtheorem{remark}[theorem]{Remark}

\usepackage[affil-it]{authblk}
\usepackage{booktabs}   

\title{Spectral theory of the non-backtracking Laplacian for graphs}
\author[1,2]{Jürgen Jost}
\author[1,3]{Raffaella Mulas\thanks{r.mulas@vu.nl}}
\author[1]{Leo Torres\thanks{leo@leotrs.com}}
\affil[1]{Max Planck Institute for Mathematics in the Sciences, Leipzig, Germany}
\affil[2]{Santa Fe Institute for the Sciences of Complexity, Santa Fe, New Mexico, USA}
\affil[3]{Vrije Universiteit Amsterdam, Amsterdam, The Netherlands}
\date{}

\begin{document}

\maketitle

\begin{abstract}We introduce a non-backtracking Laplace operator for graphs and we investigate its spectral properties. With the use of both theoretical and computational techniques, we show that the spectrum of this operator captures several structural properties of the graph in a more precise way than the classical operators that have been studied so far in the literature, including the non-backtracking matrix.

\vspace{0.2cm}
\end{abstract}

\section{Introduction}
 
Spectral graph theory is the study of the structural properties of a graph that can be inferred by the \emph{spectrum}, i.e., the multiset of the eigenvalues, of an operator associated with it. Classically, for a simple graph $G=(V,E)$ on $N$ nodes and $M$ edges, three $N \times N$ matrices are considered in the context of spectral theory: 
\begin{itemize}
    \item The \emph{adjacency matrix} $A$,
    \item The \emph{Kirchhoff Laplacian} $K \coloneqq D-A$, where $D$ denotes the diagonal degree matrix of $G$, and
    \item The \emph{random walk Laplacian} $L\coloneqq\id-D^{-1}A$, where $\id$ denotes the $N \times N$ identity matrix.
\end{itemize}
There is a vast literature on the spectral theory of all of these matrices (we refer to \cite{chung} and \cite{brouwer} for two monographs), as well as on its applications to data analysis and dynamical systems. The advantage of studying their eigenvalues is that, with little computational effort, one can infer structural properties of the graph. Moreover, since they all have real eigenvalues, the study of their theoretical results is easier than that of other operators.\newline

A common alternative to the above matrices, which does not necessarily have real eigenvalues but has nevertheless several advantages and has many properties in common with symmetric matrices, is the $2M\times 2M$ \emph{non-backtracking matrix} of $G$. It is defined as follows. Fix an arbitrary \emph{orientation} for each edge of $G$, i.e., given $e \in E$, let one of its endpoints be its \emph{input}, denoted by $\inp(e)$, and the other one be its \emph{output}, denoted by $\out(e)$, with respect to this orientation. Let then $e_1,\ldots,e_M$ denote the edges with this fixed orientation, and let $$e_{M+1}\coloneqq e^{-1}_1,\ldots,e_{2M}\coloneqq e_M^{-1}$$
denote the edges with the inverse orientation. Let then $B$ be the $2M\times 2M$ matrix with $(0,1)$--entries such that
\begin{equation*}
    B_{ij}=1\iff \out(e_i)=\inp(e_j) \text{ and }  \inp(e_i)\neq\out(e_j).
\end{equation*}
The \emph{non-backtracking matrix} of $G$ is $B^\top$, the transpose matrix of $B$. This matrix was first introduced by Hashimoto \cite{hashimoto1989zeta}, and its relationship with the theory of graph-theoretical zeta functions is well-known \cite{terras2010zeta}, as are its applications to spectral clustering \cite{bordenave2018nonbacktracking,krzakala2013spectral}, node centrality \cite{arrigo2018non,martin2014localization,torres2021nonbacktracking}, and spreading dynamics and percolation on graphs \cite{castellano2018relevance,shrestha2015message}, to name a few.
\newline

Here we introduce a new operator that, as we shall see, has many advantages in the context of spectral theory. We construct it as follows. We first define the \emph{non-backtracking graph} $\mathcal{NB}(G)$ of $G$ as the directed graph on $2M$ nodes whose adjacency matrix is $B$. We let $\mathcal{L}$ be the random walk Laplacian of $\mathcal{NB}(G)$, and we call it the \emph{non-backtracking Laplacian} of $G$.\newline

The paper is structured as follows. In Section \ref{section:basic} we give the basic definitions and we fix the main notations. In Section \ref{section:firstp} we prove some first properties on non-backtracking random walks, the non-backtracking graph, and the non-backtracking Laplacian. In Section \ref{section:gap} we prove that the spectral gap from $1$ for the non-backtracking Laplacian of a graph $G$ is always bounded below by $1/(\Delta-1)$, where $\Delta$ is the largest vertex degree in the graph, and we show that this bound is sharp. In Section \ref{section:cycles} we investigate the signature that the presence of cycles in the graph can leave in the spectrum of the non-backtracking Laplacian, and finally, in Section \ref{section:cospectrality}, we discuss cospectrality of graphs with respect to several operators. Our computations for graphs with small number of nodes suggest, in particular, that the non-backtracking Laplacian has nicer cospectrality properties than all other operators.\newline

A remark on \emph{notation}: The objects associated to a simple graph $G$ are denoted by standard capital letters, and the objects associated to a directed graph $\mathcal{G}$ are denoted by the corresponding calligraphic letters. Thus, for instance, $V$ is the vertex set of $G$, and $\mathcal{V}$ is the vertex set of $\mathcal{G}$. Moreover, we shall often use $\mathcal{G}$, instead of $\mathcal{NB}(G)$, to denote the non-backtracking Laplacian of $G$.

\section{Basic definitions and notations}\label{section:basic}
Throughout the paper we fix a simple graph $G=(V,E)$ on $N$ nodes and $M$ edges. If two vertices $v,w\in V$ are connected by an edge, we write $v\sim w$, or equivalently $w\sim v$, and we denote such edge by $(v,w)$, or equivalently by $(w,v)$.
\begin{definition}The \emph{degree} of a vertex $v$, denoted $\deg_G v$ (or simply $\deg v$, if the graph is clear by the context) is the number of edges in which it is contained. The \emph{degree matrix} of $G$ is the $N\times N$ diagonal matrix $D \coloneqq D(G) \coloneqq (D_{vw})_{v,w\in \mathcal{V}}$ whose diagonal entries are 
\begin{equation*}
    D_{vv}\coloneqq\deg v.
\end{equation*}
The \emph{adjacency matrix} of $G$ is the $N\times N$ matrix $A \coloneqq A(G) \coloneqq (A_{vw})_{v,w\in V}$ defined by
\begin{equation*}
    A_{vw} \coloneqq \begin{cases}1 & \text{if }v\sim w\\
    0 & \text{otherwise}.
    \end{cases}
\end{equation*}
\end{definition}

We assume, from here on, that $G$ has minimum degree $\geq 2$. The fact that there are no vertices of degree $0$ allows us to give the following definition:

\begin{definition}
A \emph{random walk} on $G$ is a discrete-time Markov process on $V$ such that the probability of going from a vertex $v$ to a vertex $w$ is
\begin{equation*}
    \mathbb{P}(v\rightarrow w)=\begin{cases}\frac{1}{\deg v} & \text{if }v\sim w\\
    0 & \text{otherwise.}
    \end{cases}
\end{equation*}
\end{definition}

\begin{definition}
The \emph{random walk Laplacian} of $G$ is the $N\times N$ matrix
\begin{equation*}
   L(G)\coloneqq\id-D^{-1}A,
\end{equation*}where $\id$ denotes the $N\times N$ identity matrix.
\end{definition}
\begin{remark} Since for two distinct vertices $v,w\in V$, the entry $(D^{-1}A)_{vw}$ of the rescaled adjacency matrix is precisely the probability of going from $v$ to $w$ with a random walk on $G$. the operator $L(G)$ is the generator of the random walk of $G$.
\end{remark}

Choosing an \emph{orientation} for an edge means letting one of its endpoints be its \emph{input} and the other one be its \emph{output}. We let $e=[v, w]$ denote the oriented edge whose input is $v$ and whose output is $w$. In this case, we write $\inp(e) \coloneqq v$ and $\out(e) \coloneqq w$. Moreover, we let $e^{-1} \coloneqq [w, v]$.\newline
From now on, we fix an orientation for each edge of $G$. We let $e_1,\ldots,e_M$ denote the edges of $G$ with this fixed orientation and we let 
$$e_{M+1} \coloneqq e^{-1}_1,\ldots,e_{2M} \coloneqq e_M^{-1}$$
denote the edges with the inverse orientation.\newline

The assumption that $G$ has minimum degree $\geq 2$ enables us to formulate the following definition:

\begin{definition}
A \emph{non-backtracking random walk} on $G$ is a discrete-time Markov process on the oriented edges such that the probability of going from $e_i$ to $e_j$ is
\begin{equation*}
    \mathbb{P}(e_i\rightarrow e_j)=\begin{cases}\frac{1}{\deg (\out(e_i))-1} & \text{if }\out(e_i)=\inp(e_j) \text{ and }  \inp(e_i)\neq\out(e_j)\\
    0 & \text{otherwise.}
    \end{cases}
\end{equation*}

\end{definition}

Equivalently, a non-backtracking random walk on $G=(V,E)$ can also be seen as a process on $V$, in which the probability that a random walker goes from a vertex $v$ to a vertex $w$ depends on where she was before arriving at $v$. However, this process is not Markovian, which is why it is convenient to study it from the point of view of the oriented edges.

\begin{definition}
The matrix $B\coloneqq B(G)$ is the $2M\times 2M$ matrix with $(0,1)$--entries such that
\begin{equation*}
    B_{ij}=1\iff \out(e_i)=\inp(e_j) \text{ and }  \inp(e_i)\neq\out(e_j).
\end{equation*}
The \emph{non-backtracking matrix} of $G$ is $B^\top$, the transpose matrix of $B$.
\end{definition}

From now on, we also fix a directed graph $\mathcal{G}=(\mathcal{V},\mathcal{E})$ on $\mathcal{N}$ nodes and $\mathcal{M}$ edges that has no vertices of outdegree $0$. If $\mathcal{G}$ has an edge from a vertex $v$ to a vertex $w$, we write $v\rightarrow w$ and we denote such an edge by $(v\rightarrow w)$.
\begin{remark}
Note that, although both oriented edges and directed edges are defined as ordered pairs of vertices, these two definitions are conceptually different. In fact, while directions are intrinsic of the chosen graph, orientations are not. This is what motivates us to use two different notations for oriented and directed edges.
\end{remark}

\begin{definition}
The \emph{degree} of a vertex $v$ is
\begin{equation*}
   \deg v\coloneqq \deg_{\mathcal{G}} v\coloneqq|\{w\in \mathcal{V}:(v\rightarrow w)\in \mathcal{E}\}|.
\end{equation*}The \emph{degree matrix} of $\mathcal{G}$ is the $\mathcal{N}\times \mathcal{N}$ diagonal matrix $\mathcal{D} \coloneqq \mathcal{D}(\mathcal{G}) \coloneqq (\mathcal{D}_{vw})_{v,w\in \mathcal{V}}$ whose diagonal entries are 
\begin{equation*}
    \mathcal{D}_{vv}\coloneqq\deg v.
\end{equation*}
The \emph{adjacency matrix} of $\mathcal{G}$ is the $\mathcal{N}\times \mathcal{N}$ matrix $\mathcal{A}\coloneqq\mathcal{A}(\mathcal{G})\coloneqq(\mathcal{A}_{vw})_{v,w\in \mathcal{V}}$ defined by
\begin{equation*}
    \mathcal{A}_{vw}\coloneqq\begin{cases}1 & \text{if }v\rightarrow w\\
    0 & \text{otherwise}.
    \end{cases}
\end{equation*}
\end{definition}
\begin{remark}
Note that the degree counts the number of outgoing edges from a node, and therefore it is also often called the outdegree.
\end{remark}
\begin{definition}
A \emph{random walk} on $\mathcal{G}$ is a discrete-time Markov process on $\mathcal{V}$ such that the probability of going from a vertex $v$ to a vertex $w$ is
\begin{equation*}
    \mathbb{P}(v\rightarrow w)=\begin{cases}\frac{1}{\deg v} & \text{if }v\rightarrow w\\
    0 & \text{otherwise.}
    \end{cases}
\end{equation*}
\end{definition}

\begin{definition}
The \emph{random walk Laplacian} of $\mathcal{G}$ is the $\mathcal{N}\times \mathcal{N}$ matrix
\begin{equation*}
   \mathcal{L}(\mathcal{G})\coloneqq\id-\mathcal{D}^{-1}\mathcal{A}.
\end{equation*}
\end{definition}

\begin{remark} As in the undirected case, the  random walk Laplacian of $\mathcal{G}$ is the generator of the random walk on $\mathcal{G}$.
\end{remark}

\begin{definition}
The \emph{non-backtracking graph} of $G$ is the directed graph $\mathcal{NB}(G)$ on vertices $e_1,\ldots,e_{2M}$, that has $B$ as adjacency matrix.
\end{definition}

The non-backtracking graph $\mathcal{NB}(G)$ has been studied before in the literature, for instance in \cite{Kempton2016} and \cite{fasino2021hitting}, where it is called the \emph{Hashimoto graph}.

\begin{example}\label{ex:cycle}
If $G$ is the cycle graph on $N$ nodes, then $\mathcal{NB}(G)$ is given by two disconnected directed cycles on $N$ nodes (Figure \ref{fig:cycles1}).
\end{example}

\begin{figure}[h]
    \centering
    \includegraphics[width=12cm]{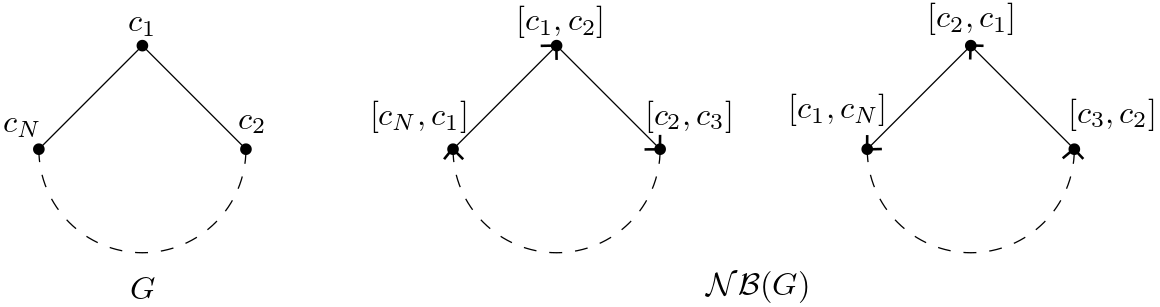}
    \caption{The cycle graph $G$ and its non-backtracking graph $\mathcal{NB}(G)$.}
    \label{fig:cycles1}
\end{figure}

Clearly, a random walk on the directed graph $\mathcal{NB}(G)$ is equivalent to a non-backtracking random walk on $G$. Moreover, for each oriented edge $e_i$,
\begin{equation*}
    \deg_{\mathcal{G}} e_i=\deg_G(\out(e_i))-1.
\end{equation*}In particular, the assumption that $G$ has minimum degree $\geq 2$ implies that $\mathcal{NB}(G)$ has minimum degree $\geq 1$. Hence, the random walk Laplacian of $\mathcal{NB}(G)$ is well-defined.\newline Also, as a consequence of the above equality, we have that $G$ is $k+1$--\emph{regular} (meaning that all its vertices have constant degree $k+1$) if and only if $\mathcal{NB}(G)$ is $k$--regular. Similarly, $G$ is bipartite if and only if $\mathcal{NB}(G)$ is bipartite, since $G$ has odd-length cycles if and only if $\mathcal{NB}(G)$ has odd-length cycles.

\begin{definition}
The \emph{non-backtracking Laplacian} of $G$, denoted by $\mathcal{L}\coloneqq\mathcal{L}(G)$, is the random walk Laplacian $\mathcal{L}(\mathcal{NB}(G))$ of $\mathcal{NB}(G)$.
\end{definition}

Note that the non-backtracking Laplacian of $G$ can be written as $\id-\mathcal{T}(G)$, where $\mathcal{T}(G)$ is the transition matrix of the non-backtracking process. This matrix is spectrally equivalent to $\mathcal{L}(G)$ (in fact, their eigenvalues are the same, up to a translation) and it has been considered in \cite{fasino2021hitting}. However, in \cite{fasino2021hitting}, the spectral properties of $\mathcal{T}(G)$ (and therefore also of $\mathcal{L}(G)$) have not been investigated. Moreover, to the best of our knowledge, the above definition of $\mathcal{L}(G)$ has not been considered before in the literature. In particular, this operator is never symmetric, therefore it does not coincide with the symmetric Laplacian that has been studied in \cite{Alon2007,Kempton2016} for non-backtracking random walks.

\section{First properties}\label{section:firstp}


\subsection{Non-backtracking walks and graph}

We now investigate the first properties of the non-backtracking walks and non-backtracking graph of a simple graph. We start with the following observation.
\begin{remark}
As anticipated in the previous section, although we choose to see a non-backtracking walk on $G=(V,E)$ as a process on the oriented edges, this can be equivalently seen as a process on $V$. However, in this case the process is not Markovian, implying that it is much more complicated. To see this, we consider a non-backtracking walk which is seen as a process on $V$, and we let $\mathbb{P}_n(v_0 \rightarrow v_n)$ denote the probability of going from vertex $v_0$ to vertex $v_n$ with a non-backtracking random walk of length $n$ which is starting at $v_0$. In this case,
\begin{equation*}
      \mathbb{P}_1(v_0 \rightarrow v_1)=\frac{A_{v_0v_1}}{\deg v_0};
\end{equation*}
\begin{equation*}     
    \mathbb{P}_2(v_0 \rightarrow v_2)=\begin{cases}\sum_{v_1\in N(v_0)\cap N(v_2)}\frac{1}{\deg v_0}\cdot \frac{1}{\deg v_1-1} & \text{if }v_0\neq v_2,\\
    0 & \text{if }v_0=v_2;
    \end{cases}
\end{equation*}
\begin{align*}
      \mathbb{P}_3(v_0 \rightarrow v_3)=\sum\frac{1}{\deg v_0-A_{v_0v_3}}\cdot \frac{1}{|N(v_1)\cap N(v_3)|-A_{v_3v_0}}\cdot \frac{1}{\deg v_2-1},
      \end{align*}where the sum is over $v_1\in N(v_0)\setminus \{v_3\}$, $v_2 \in N(v_1)\cap N(v_3)\setminus\{v_0\}$;
      \begin{align*}
       \mathbb{P}_4(v_0 \rightarrow v_4)=\sum &\frac{1}{\deg v_0}\cdot \frac{1}{\deg v_1-1-A_{v_1v_4}}\cdot \frac{1}{| N(v_2)\cap N(v_4)|-A_{v_4v_1}}\\ &\cdot \frac{1}{\deg v_3-1},
\end{align*}where the sum is over
\begin{align*}
    &v_1\in N(v_0)\\ 
    &v_2 \in N(v_1)\setminus\{v_0,v_4\}\\ 
    &v_3\in N(v_2)\cap N(v_4)\setminus \{v_1\}.
\end{align*}
 Moreover, for $n\geq 5$,
\begin{align*}
 \mathbb{P}_n(v_0 \rightarrow v_n)= \sum
  &\frac{1}{\deg v_0}\cdot \prod_{i=2}^{n-3}\frac{1}{\deg v_{i-1}-1}\cdot \frac{1}{\deg v_{n-3}-1-A_{v_{n-3}v_n}}\\
  & \cdot\frac{1}{|N(v_{n-2})\cap N(v_n)|-A_{v_{n-3}v_n}}\cdot \frac{1}{\deg v_{n-1}-1},
\end{align*}where the above sum is over
\begin{align*}
    &v_1\in N(v_0)\\ &v_i\in N(v_{i-1})\setminus\{v_{i-2}\} \text{ for }i=2,\ldots,n-3\\ &v_{n-2}\in N(v_{n-3})\setminus \{v_{n-4},v_n\}\\ &v_{n-1}\in N(v_{n-2})\cap N(v_n)\setminus \{v_{n-3}\}.
\end{align*}Hence, despite having a simple description, a non-backtracking walk on $G$ is highly complex, if seen as a process on $V$.
\end{remark}

We now investigate some structural properties of the non-backtracking graph of $G=(V,E)$. The first one is the following.

\begin{lemma}
Let $G=(V,E)$ be a simple graph on $N$ nodes and $M$ edges, with minimum degree $\geq 2$, and let $\mathcal{G}=(\mathcal{V},\mathcal{E})$ be its non-backtracking graph. Then, $\mathcal{G}$ has $2M$ nodes and $\sum_{v\in V}(\deg_{G}v)^2-2M$ edges.
\end{lemma}
\begin{proof}
The fact that $\mathcal{G}$ has $2M$ nodes is clear by definition. Now, observe that, by definitions of vertex degrees for $G$ and $\mathcal{G}$, we have that
\begin{equation*}
    |\mathcal{E}|=\sum_{e\in \mathcal{V}}\deg_{\mathcal{G}}e, \mbox{ while } M=\frac{\sum_{v\in V}\deg_G v}{2}.
\end{equation*}Therefore,
\begin{equation*}
   |\mathcal{E}|= \sum_{e\in \mathcal{V}}\deg_{\mathcal{G}}e=\sum_{v\in V}\deg_{G}v\cdot (\deg_{G}v-1)=\sum_{v\in V}(\deg_{G}v)^2-2M,
\end{equation*}where the second equality is due to the fact that, for each $v\in V$, there are $\deg_G v$ oriented edges of the form $[w,v]$ and, for each of these oriented edges, there are $\deg_G v-1$ oriented edges of the form $[v,z]$ with $z\neq w$.
\end{proof}
Now, in Example \ref{ex:cycle} we saw that, if $G$ is a cycle graph, then its non-backtracking graph is given by two disconnected cycles. The next theorem shows that this is the only case in which a connected simple graph has a disconnected non-backtracking graph.
\begin{theorem}\label{thm:connected}
Let $G=(V,E)$ be a simple connected graph on $N$ nodes and $M$ edges, with minimum degree $\geq 2$, and let $\mathcal{G}=(\mathcal{V},\mathcal{E})$ be its non-backtracking graph. Then, the following are equivalent:
\begin{enumerate}
    \item $G$ is not the cycle graph;
    \item $G$ has at least two cycles;
    \item $\mathcal{G}$ is weakly connected;
    \item $\mathcal{G}$ is strongly connected.
\end{enumerate}
\end{theorem}
\begin{proof} Clearly, since the minimum degree in $G$ is $\geq 2$, the first two conditions are equivalent to each other. Moreover, 4 clearly implies 3 and, by Example \ref{ex:cycle}, 3 implies 1. Hence, if we prove that 1 implies 4, we are done.\newline

Assume that $G$ is not the cycle graph and fix two distinct elements $[v,w],[x,y]\in \mathcal{V}$. We want to show that there exists a directed path, in $\mathcal{G}$, from $[v,w]$ to $[x,y]$.\newline
Since $G$ is connected, there exists a path, in $G$, of the form
\begin{equation*}
    (w,p_1),(p_1,p_2),\ldots, (p_{k-1},p_k),(p_k,x)
\end{equation*}which is non-backtracking, i.e., such that $p_2\neq w$, $p_{k-1}\neq x$ and $p_i\neq p_{i+2}$ for $i=1,\ldots,k-2$. This gives a directed path, in $\mathcal{G}$, of the form
\begin{equation*}
 [w,p_1]\rightarrow[p_1,p_2]\rightarrow\ldots\rightarrow [p_{k-1},p_k]\rightarrow[p_k,x].  
\end{equation*}
If $p_1\neq v$ and $p_k\neq y$, then $\mathcal{G}$ has also the directed path
\begin{equation*}
 [v,w]\rightarrow [w,p_1]\rightarrow[p_1,p_2]\rightarrow\ldots\rightarrow [p_{k-1},p_k]\rightarrow[p_k,x]\rightarrow [x,y],
\end{equation*}hence the claim holds in this case.\newline
If $p_1= v$ and $p_k\neq y$, then by the assumptions that $G$ is not the cycle graph and each vertex in $G$ has degree $\geq 2$, it follows that there exists a non-backtracking path in $G$ of the form (Figure \ref{fig:connected})
\begin{equation*}
  (v,w),(w=q_1,q_2),\ldots, (q_{r-1},q_r=c_1),\ldots, (c_{s-1},c_s=c_1),
\end{equation*}
for some $r\geq 1$ and $s\geq 3$. 

\begin{figure}[h]
    \centering
    \includegraphics[width=8cm]{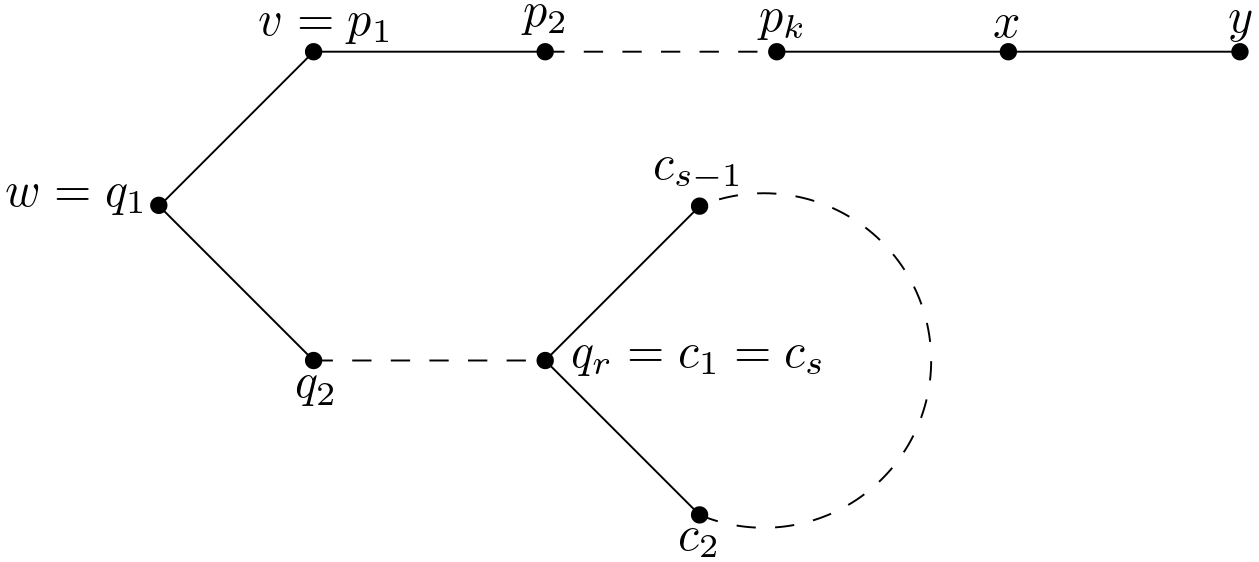}
    \caption{An illustration of the proof of Theorem \ref{thm:connected}.}
    \label{fig:connected}
\end{figure}

In this case, there exists a directed path from $[v,w]$ to $[x,y]$, in $\mathcal{G}$, of the form
\begin{align*}
    &[v,w]\rightarrow [w,q_2]\rightarrow\ldots\rightarrow [q_{r-1},q_r=c_1]\rightarrow  \ldots\rightarrow  [c_{s-1},c_s=c_1]\\
    &\rightarrow[c_1=q_r,q_{r-1}]\rightarrow \ldots \rightarrow[q_2,w]\rightarrow [w,v]\\
    &\rightarrow [v,p_2]\rightarrow \ldots \rightarrow [p_k,x]\rightarrow [x,y].
\end{align*}
If either $p_1\neq v$ and $p_k= y$, or $p_1\neq v$ and $p_k\neq y$, the claim follows in a similar way. This shows that, if $G$ is not the cycle graph, then $\mathcal{G}$ is strongly connected.
\end{proof}

\subsection{Non-backtracking Laplacian}
In this section, we shall focus on the non-backtracking Laplacian of $G$. We start by considering a general directed graph $\mathcal{G}=(\mathcal{V},\mathcal{E})$ on $\mathcal{N}$ nodes and $\mathcal{M}$ edges that has no vertices of degree $0$ and which is not, necessarily, the non-backtracking graph of a simple graph. We observe that, given any $\mathcal{N}\times \mathcal{N}$ real matrix $\mathcal{O}$, we can equivalently see it as an operator 
\begin{equation*}
    \mathcal{O}:\{f:\mathcal{V}\rightarrow \mathbb{C}\}\rightarrow \{f:\mathcal{V}\rightarrow \mathbb{C}\}.
\end{equation*}In particular, an eigenvector $\mathbf{x}\in \mathbb{C}^{\mathcal{N}}$ for $\mathcal{O}$ with eigenvalue $\lambda\in \mathbb{C}$ can be seen as a function (called an \emph{eigenfunction}) $f:\mathcal{V}\rightarrow \mathbb{C}$ such that
\begin{equation*}
    \mathcal{O}f(\omega)=\lambda f(\omega), \quad \forall \omega\in \mathcal{V}.
\end{equation*}This applies, in particular, to the random walk Laplacian $\mathcal{L}$ of $\mathcal{G}$, which can be seen as an operator such that, given $f:\mathcal{V}\rightarrow \mathbb{C}$ and $\omega \in \mathcal{V}$,
\begin{equation*}
    \mathcal{L}f(\omega)=f(\omega)-\frac{1}{\deg \omega}\left( \sum_{\omega\rightarrow \tau}f(\tau)\right).
\end{equation*}In particular, a pair $(\lambda,f)$ with $\lambda\in \mathbb{C}$ and $f:\mathcal{V}\rightarrow \mathbb{C}$ is an eigenpair for $\mathcal{L}$ if and only if, for each $\omega\in \mathcal{V}$,
\begin{equation*}
    (1-\lambda)f(\omega)=\frac{1}{\deg \omega}\left( \sum_{\omega\rightarrow \tau}f(\tau)\right).
\end{equation*}The next results hold for the random walk Laplacian $\mathcal{L}$ and adjacency matrix $\mathcal{A}$ of any such directed graph $\mathcal{G}$. The observations in Remark \ref{rmk:Bauer} below have been proved by Bauer in \cite{Bauer}.

\begin{remark}\label{rmk:Bauer}
Clearly, for any directed graph $\mathcal{G}=(\mathcal{V},\mathcal{E})$ on $\mathcal{N}$ nodes, its random walk Laplacian $\mathcal{L}$ has $\mathcal{N}$ eigenvalues (counted with algebraic multiplicity) that sum to $\mathcal{N}$, since $\mathcal{L}$ is an $\mathcal{N}\times \mathcal{N}$ matrix that has trace $\mathcal{N}$. Moreover, by Proposition 3.2 in \cite{Bauer}, the spectrum of $\mathcal{L}$ is contained in the complex disc $D(1,1)$. In particular, the real eigenvalues are contained in $[0,2]$. Also, by Proposition 3.1 in \cite{Bauer}, $0$ is an eigenvalue for $\mathcal{L}$ and the constant functions $f:\mathcal{V}\rightarrow\mathbb{C}$ are the corresponding eigenfunctions. As a consequence, from the spectrum of $\mathcal{L}$ we can derive the number of connected components of $\mathcal{G}$. Notably, this does not hold for the adjacency matrix $\mathcal{A}$ of $\mathcal{G}$. 
\end{remark}

Now, it is well-known that, for a simple undirected graph $G$, the following are equivalent \cite{chung,brouwer}:
\begin{enumerate}
    \item $G$ is bipartite;
    \item The spectrum of the random walk Laplacian $L(G)$ is symmetric with respect to the line $x=1$;
    \item $2$ is an eigenvalue of $L(G)$;
    \item The spectrum of the adjacency matrix $A(G)$ is symmetric with respect to the line $x=0$.
\end{enumerate}
In the next proposition we prove that, for a directed graph $\mathcal{G}$, condition 1 above implies 2, 3 and 4. However, as shown in Example \ref{ex:bipartite} below, these conditions are not equivalent.
\begin{proposition}\label{prop:bipartite}
If $\mathcal{G}$ is a directed bipartite graph, then the spectra of both its random walk Laplacian $\mathcal{L}$ and its adjacency matrix $\mathcal{A}$ are symmetric. Hence, in particular, $2$ is an eigenvalue for $\mathcal{L}$.
\end{proposition}
\begin{proof}
Without loss of generality, we only prove the first claim for the random walk Laplacian, the other case being similar. The proof follows the same idea as the proof for the undirected case in \cite[Lemma 1.8]{chung}. \newline
If $\mathcal{G}=(\mathcal{V},\mathcal{E})$ is bipartite, $\mathcal{V}_1\sqcup \mathcal{V}_2$ is a corresponding bipartition and $(\lambda,f)$ is an eigenpair for $\mathcal{L}$, then also $(2-\lambda,\tilde{f})$ is an eigenpair, where
\begin{equation*}
    \tilde{f}\coloneqq\begin{cases}f &\text{ on }\mathcal{V}_1\\ -f&\text{ on }\mathcal{V}_2. \end{cases}
\end{equation*}As an immediate consequence, $2$ is an eigenvalue for $\mathcal{L}$, since $0$ is always an eigenvalue.
\end{proof}

The next example shows a directed graph that is not bipartite but is such that $2$ is an eigenvalue for its random walk Laplacian.
\begin{example}\label{ex:bipartite}
Consider the connected graph $\mathcal{G}$ in Figure \ref{fig:counterex}, where the numbers on the vertices indicate the values of a function $f$. Then, $\mathcal{G}$ is not bipartite, and $f$ is an eigenfunction for $\mathcal{L}(G)$ with eigenvalue $2$.
\end{example}
\begin{figure}[h]
    \centering
    \includegraphics[width=0.35\textwidth]{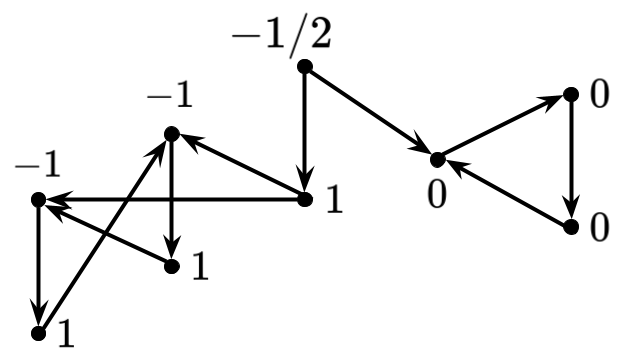}
    \caption{The directed graph in Example \ref{ex:bipartite}.}
    \label{fig:counterex}
\end{figure}

\begin{remark}\label{rmk:regular}
For any directed graph $\mathcal{G}$, its random walk Laplacian $\mathcal{L}$ and its adjacency matrix $\mathcal{A}$ satisfy:
\begin{itemize}
    \item $\mathbf{x}$ is an eigenvector for $\mathcal{A}$ with eigenvalue $0$ if and only if $\mathbf{x}$ is an eigenvector for $\mathcal{L}$ with eigenvalue $1$;
    \item If $\mathcal{G}$ is $k$--regular, then $(\mathbf{x},\lambda)$ is an eigenpair for $\mathcal{A}$ if and only if $(\mathbf{x},1-\frac{\lambda}{k})$ is an eigenpair for $\mathcal{L}$.
\end{itemize}
\end{remark}

The properties of the random walk Laplacian that we investigated so far in this section hold for any directed graph. For the next observations and results, we focus on the case of non-backtracking graphs. As before, we fix a simple graph $G=(V,E)$ on $N$ nodes and $M$ edges that has minimum degree $\geq 2$. For simplicity, we denote its non-backtracking graph by $\mathcal{G}=(\mathcal{V},\mathcal{E})$. We let $\mathcal{A}$ denote the adjacency matrix of $\mathcal{G}$ (equivalently, the transpose of the non-backtracking matrix of $G$) and we let $\mathcal{L}$ denote the random walk Laplacian of $\mathcal{G}$ (equivalently, the non-backtracking Laplacian of $G$). Similarly, we let $A$ and $L$ denote the adjacency matrix and the random walk Laplacian of $G$, respectively. 

\begin{remark}\label{rmk:regular2}
As observed in Section \ref{section:basic}, $G$ is regular if and only if $\mathcal{G}$ is regular.  In this case, by Remark \ref{rmk:regular}, the spectral properties of $\mathcal{L}$ and $\mathcal{A}$ are equivalent to each other, and similarly also the spectral properties of $L$ and $A$ are equivalent to each other. But since it is known that, in the regular case, the eigenvalues of $\mathcal{A}$ can be recovered by those of $A$ \cite{kotani20002,Kempton2016}, it follows that in this case also the spectral theory of $\mathcal{L}$ can be recovered from the one of $A$ or, equivalently, of $L$.
\end{remark}

Before stating the final theorem of this section, we define the $2M\times 2M$ matrix \begin{equation*}
    P\coloneqq\begin{pmatrix}
  \begin{matrix}
  0 & \id \\
\id & 0
  \end{matrix}
\end{pmatrix}
\end{equation*}and we observe that $P^\top=P$, while $P^2=\id$.\newline
Moreover, given $\mathbf{x},\mathbf{y}\in \mathbb{C}^{2M}$, we let $\langle \mathbf{x},\mathbf{y} \rangle \coloneqq \overline{\mathbf{x}}^\top \mathbf{y}$ be the usual complex inner product, and we define their \emph{$P$-product} as
\begin{equation*}
    (\mathbf{x},\mathbf{y})_P \coloneqq \langle \mathbf{x}, P \mathbf{y} \rangle = \overline{\mathbf{x}}^\top P \mathbf{y}.
\end{equation*}

\begin{theorem}\label{thm:Psymm}
Let $\mathcal{L}$ be the non-backtracking Laplacian of a graph. Then,
\begin{enumerate}
    \item $\mathcal{L}^\top=P\mathcal{L}P$;
    \item $\mathcal{L}$ is self-adjoint with respect to the $P$-product;
    \item If $(\lambda,\mathbf{x})$ is an eigenpair for $\mathcal{L}$ and $\lambda\in\mathbb{C}$ is not real, then  
    \begin{equation*}
    \sum_{[v, w]}\overline{x_{[v, w]}}\cdot x_{[w, v]}=\sum_{i=1}^m\left(\overline{x_i}\cdot x_{i+M}+\overline{x_{i+M}}\cdot x_{i}\right)=0.
\end{equation*}
\end{enumerate}
\end{theorem}
\begin{proof}
\begin{enumerate}
    \item We have that, for $i\neq j$ with $i, j \leq M$,
\begin{equation*}
    \mathcal{L}_{ij}=-\mathbb{P}(e_i \rightarrow e_j)=-\mathbb{P}(e_j^{-1} \rightarrow e_i^{-1})=\mathcal{L}_{j+M,i+M}.
\end{equation*}This allows us to write
\begin{equation*}\label{eqn:plp}
    \mathcal{L}^\top = P \mathcal{L} P.
\end{equation*}
\item Since $P^\top=P$ and $P^2=\id$, we have that
\begin{align*}
    (\mathbf{x},\mathcal{L}\mathbf{y})_P&=\langle  \mathbf{x}, P \mathcal{L}\mathbf{y}\rangle =\langle  P \mathbf{x},  \mathcal{L}\mathbf{y}\rangle =\langle   \mathcal{L}^\top P\mathbf{x},  \mathbf{y}\rangle \\
    &=\langle  P \mathcal{L} P^2\mathbf{x},  \mathbf{y}\rangle =\langle  P \mathcal{L} \mathbf{x},  \mathbf{y}\rangle =\langle   \mathcal{L} \mathbf{x}, P \mathbf{y}\rangle =(\mathcal{L}\mathbf{x},\mathbf{y})_P.
\end{align*}
Therefore, $\mathcal{L}$ is self-adjoint with respect to the $P$-product.
\item The second claim implies that, if $(\lambda,\mathbf{x})$ is an eigenpair for $\mathcal{L}$, then
\begin{equation*}
   \lambda(\mathbf{x},\mathbf{x})_P=(\mathbf{x},\lambda \mathbf{x})_P= (\mathbf{x},\mathcal{L}\mathbf{x})_P=(\mathcal{L}\mathbf{x},\mathbf{x})_P=(\lambda\mathbf{x}, \mathbf{x})_P= \overline{\lambda}(\mathbf{x},\mathbf{x})_P.
\end{equation*}
Hence, if $\lambda\neq\overline{\lambda}$, i.e., $\lambda$ is not real, then $(\mathbf{x},\mathbf{x})_P=0$, that is, $\overline{\mathbf{x}}^\top P \mathbf{x}=0$, which can be re-written as
\begin{equation*}
   \sum_{[v, w]}\overline{x_{[v, w]}}\cdot x_{[w, v]}==\sum_{i=1}^m\left(\overline{x_i}\cdot x_{i+M}+\overline{x_{i+M}}\cdot x_{i}\right)=0.
\end{equation*}
\end{enumerate}
\end{proof}

\begin{remark}
The preceding theorem also holds for the adjacency matrix of $\mathcal{G}$ (equivalently, the non-backtracking matrix of $G$). The proofs are analogous and thus omitted.
\end{remark}

Following the terminology in \cite{bordenave}, the first condition in Theorem \ref{thm:Psymm} can be reformulated by saying that $\mathcal{L}$ is \emph{PT-symmetric} (where PT stands for parity-time). Moreover, following the terminology in \cite{gohberg2006indefinite}, the second condition in Theorem \ref{thm:Psymm} can be reformulated by saying that $\mathcal{L}$ is \emph{$P$-self adjoint}.

\begin{remark}
Figure \ref{fig:wow} shows the eigenvalues of the non-backtracking matrix and the non-backtracking Laplacian matrix of an Erd\H{o}s-R\'enyi graph. It is known that the ``bulk'' of the eigenvalues of the matrix $\mathcal{A}$ (equivalently, of the non-backtracking matrix $\mathcal{A}^\top$) concentrates on the circle with radius $\sqrt{\alpha - 1}$, where $\alpha$ is the expected average degree of the graph \cite{bordenave2018nonbacktracking,krzakala2013spectral}.
Figure \ref{fig:wow} seems to suggest, among other things, that there is a similar behavior in the case of the non-backtracking Laplacian, where most eigenvalues are observed to concentrate around a circle with radius $1/\sqrt{\alpha - 1}$.
Fully establishing this fact is left for future research.
In the case of the non-backtracking Laplacian, the figure also shows there is a spectral gap from $(1, 0)$ of size $1/(\Delta - 1)$, where $\Delta$ is the largest degree in the graph.
We establish this fact in the next section.
\end{remark}

\begin{figure}
\makebox[\textwidth][c]{\includegraphics[width=1.2\textwidth]{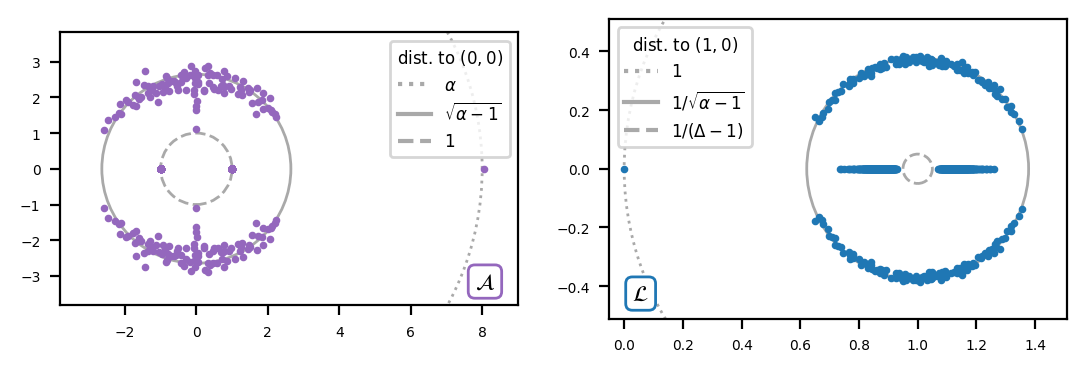}}
\caption{\label{fig:wow}
Eigenvalues of $\mathcal{A}$ (left) and $\mathcal{L}$ (right) of a random Erd\H{o}s-R\'enyi graph with $N=100$ nodes and expected average degree $\alpha=8$.
$\Delta$ is the largest degree in the graph.
}
\end{figure}

\section{Spectral gap from 1}\label{section:gap}

In this section, we investigate another property of the non-backtracking Laplacian, namely, the spectral gap from $1$.\newline

As before, we fix a simple graph $G=(V,E)$ that has minimum vertex degree $\delta \geq 2$. We let $\Delta$ denote the maximum vertex degree of $G$, and we let $\mathcal{G}$ denote the non-backtracking graph of $G$. We also let $\mathcal{A}$ be the adjacency matrix of $\mathcal{G}$ (equivalently, the transpose of the non-backtracking matrix of $G$), we let $\mathcal{D}$ denote the degree matrix of $\mathcal{G}$ and we let $\mathcal{L}=\id-\mathcal{D}^{-1}\mathcal{A}$ denote the random walk Laplacian of $\mathcal{G}$ (equivalently, the non-backtracking Laplacian of $G$).\newline Moreover, given any operator $\mathcal{O}$, we let $\sigma(\mathcal{O})$ denote its spectrum and we let $\rho(\mathcal{O})$ denote its spectral radius, i.e., the largest modulus of its eigenvalues. \newline
 
It is known that $0\in\sigma(\mathcal{A}^\top)= \sigma(\mathcal{A})$ if and only if $G$ contains nodes of degree $1$ \cite{torres2021nonbacktracking}.  Since $\delta \ge 2$, this implies that $1\notin \sigma(\mathcal{L})$. Hence, the spectral gap from $1$ for $\mathcal{L}$,
 \begin{equation*}
\varepsilon\coloneqq\min_{\lambda\in \sigma(\mathcal{L})}\left|1-\lambda\right|=\min_{\lambda\in \sigma(\mathcal{D}^{-1}\mathcal{A})}|\lambda|,
\end{equation*}is non-zero. In Theorem \ref{thm:gap1} below we give a lower bound for $\varepsilon$, and we prove that the bound is sharp.\newline

The proof of Theorem \ref{thm:gap1} uses the fact that
\begin{equation*}
\min_{\lambda \in \sigma(\mathcal{A}P)} | \lambda |^2 = 1,
\end{equation*}
where $P$ is the operator that satisfies $P\mathcal{A}P = \mathcal{A}^*$ and $P\mathcal{L}P = \mathcal{L}^*$.
This fact is established in Lemma \ref{lem:bp-spectrum}, which we relegate to the end of this section.
 
\begin{theorem}\label{thm:gap1}Let $G$ be a simple graph with maximum vertex degree $\Delta$. Then, the spectral gap from $1$ for the non-backtracking Laplacian $\mathcal{L}$ of $G$ satisfies
 \begin{equation*}
      \varepsilon \geq \frac{1}{\Delta-1}.
 \end{equation*}Moreover, the bound is sharp.
 
 \begin{proof}
 We follow the notations in the beginning of this section. We observe that, since $0\notin\sigma(\mathcal{A})$, the matrix $\mathcal{A}$ is invertible. Therefore, we can write
 \begin{equation*}
\varepsilon^{-1}=\max_{\lambda\in \sigma(\mathcal{D}\mathcal{A}^{-1})}|\lambda|=\rho(\mathcal{D}\mathcal{A}^{-1}).
 \end{equation*}
 Further, it is known that any sub-multiplicative norm $\| \cdot \|$ satisfies
\begin{equation*}
\rho \left( \mathcal{D A}^{-1} \right) \leq \| \mathcal{D A}^{-1} \| \leq \| \mathcal{D} \| \| \mathcal{A}^{-1} \|.
\end{equation*}
Take for example the spectral norm $\| \cdot \|_2$ and note that $\| \mathcal{D} \|_2=\Delta-1$.
Also note that $\| \mathcal{A}^{-1} \|_2^2$ equals the largest magnitude among the eigenvalues of $\left( \mathcal{A}^{-1} \right)^* \mathcal{A}^{-1}$.
But we have $\mathcal{A}^* = P \mathcal{A} P$, and thus
\begin{equation*}
\| \mathcal{A}^{-1}\|_2^2 = \max_{\lambda\in\sigma\left( P \mathcal{A}^{-1} P \mathcal{A}^{-1} \right)} | \lambda  | = \max_{\lambda\in\sigma\left( P \mathcal{A}^{-1} \right)} | \lambda |^2 = \frac{1}{\min_{\lambda\in\sigma\left( \mathcal{A} P \right)} |\lambda |^2} = 1,
\end{equation*}
where we have used that the smallest magnitude among eigenvalues of $\mathcal{A} P$ is $1$, this is shown in Lemma \ref{lem:bp-spectrum}. Therefore,
\begin{equation*}
\varepsilon^{-1} = \rho \left( \mathcal{D A}^{-1} \right) \leq \| \mathcal{D} \|_2 \| \mathcal{A}^{-1} \|_2 = \Delta - 1.
\end{equation*}
Thus, the spectral gap $\varepsilon$ is at least $\left( \Delta - 1 \right)^{-1}$.\newline To prove the lower bound is sharp, recall that each eigenvalue $\lambda$ of $\mathcal{A}$ satisfies $1 \leq |\lambda|$ and $1 \in \sigma(\mathcal{A})$ \cite{kotani20002}. Thus we have $\rho( \mathcal{A}^{-1} )= 1$. Now suppose $G$ is $\Delta$--regular, then
\begin{equation*}
\varepsilon^{-1} = \rho \left( \mathcal{D A}^{-1} \right) = \left( \Delta - 1 \right) \rho \left( \mathcal{A}^{-1} \right) = \left( \Delta - 1 \right),
\end{equation*}hence this bound is sharp.
\end{proof}
\end{theorem}

\begin{remark}
For the random walk Laplacian of general graphs, the result in Theorem \ref{thm:gap1} doesn't hold. For instance, in the undirected case, $1$ can be an eigenvalue, and the spectral gap from $1$ has been investigated in \cite{petalsbooks}.
\end{remark}

As promised, we now prove the fact that $\min_{\lambda \in \sigma(\mathcal{A}P)} | \lambda |^2 = 1$, which was used in the proof of Theorem \ref{thm:gap1}.
In fact, we compute every single eigenvalue and eigenfunction of $\mathcal{A}P$.
In the following Lemma we continue to assume $\delta \geq 2$ and that $G=(V,E)$ has $M$ edges and $N$ nodes, and we let $\mathcal{G} = (\mathcal{V},\mathcal{E})$ denote its non-backtracking graph.
Part of the following Lemma was stated in \cite{bordenave2018nonbacktracking} albeit without proof.

\begin{lemma}\label{lem:bp-spectrum}
The spectrum of $\mathcal{A}P$ is given by $\{ -1 \} \cup \{ \deg v-1 \}_{v \in V}$, where each positive eigenvalue $d$ has multiplicity equal to the number of nodes $v$ with $\deg v -1= d$, and the multiplicity of $-1$ equals $2M - N$.
\end{lemma}

\begin{proof}
Note that the spectrum of $\mathcal{A}P$ equals the spectrum of $P^\top\mathcal{A}^\top=P\mathcal{A}^\top$, which equals the spectrum of $\mathcal{A}^\top P$. 
From the definitions of $\mathcal{A}$ and $P$ we may compute 
\begin{equation*}
\left( \mathcal{A}^TP \right)_{i,j} = \begin{cases}
1, & \text{if } \inp(e_i) = \inp(e_j), \, \out(e_i) \neq \out(e_j), \\
0, & \text{otherwise}.
\end{cases}
\end{equation*}
With this, the eigenproblem $\mathcal{A}^\top Pf = \lambda f$ simplifies to the system of equations
\begin{equation}\label{eqn:bp-system}
\left( \lambda + 1 \right) f(e) = \sum_{\inp(e') = \inp(e)} f(e'), \text{ for each } e \in \mathcal{V}.
\end{equation}

First fix a node $v \in V$ with degree $d$ and define the function $f_v \colon \mathcal{V} \to \mathbb{R}$ by
\begin{equation*}
f_v(e) \coloneqq \begin{cases}
d, & \text{if }\inp(e) = v \\
0, & \text{otherwise}.
\end{cases}
\end{equation*}
Then $f_v$ satisfies Equation \eqref{eqn:bp-system} with eigenvalue $\lambda = d - 1$, and the set $\left\{ f_v \right\}_{v \in V}$ is linearly independent.
In other words, $\lambda = \deg v-1$ is an eigenvalue of $\mathcal{A}^\top P$, for each node $v$, and each such eigenvalue has multiplicity equal to the number of nodes $v$ with the corresponding degree.

Now fix an oriented edge $[v, w]$ with $v,w \in V$ and suppose $v$ has degree $d$.
Define the function $g_{[v, w]} \colon \mathcal{V} \to \mathbb{R}$ by
\begin{equation*}
g_{[v, w]}(e) \coloneqq \begin{cases}
-1, & \text{if }\inp(e)=v, \out(e)=w \\
\frac{1}{d-1}, & \text{if }\inp(e)=v, \out(e) \neq w \\
0, & \text{otherwise}.
\end{cases}
\end{equation*}
One can see that $g_{[v, w]}$ satisfies equations \eqref{eqn:bp-system} for eigenvalue $\lambda = -1$.
Now fix $v \in V$ and let $w_1, \ldots, w_d$ be its neighbors.
It is clear that the matrix
\begin{equation*}
\left[ g_{[v, w_1]} | \ldots | g_{[v, w_d]} \right]
\end{equation*}
has rank $d-1$.
Furthermore, for two different nodes $v$ and $v'$, the vectors ${g_{[v, w]}}$ and ${g_{[v', w']}}$ are linearly independent.
We have shown that the set $\left\{ g_{[v, w]} \right\}_{v \sim w \in E}$, which contains $2M$ elements, contains exactly ${\sum_v \left( \deg v - 1 \right)} = 2M - N$ linearly independent eigenfunctions of $\lambda = -1$, and we are done.
\end{proof}

\begin{remark}
Since by Lemma \ref{lem:bp-spectrum} all eigenvalues are $\neq 0$, we can conclude that if $f$ is an eigenfunction of $\mathcal{A}P$ with eigenvalue $\lambda$, then $Pf$ is an eigenfunction of $P\mathcal{A}$ with the same eigenvalue.
Therefore, the preceding Lemma also describes the eigenvalues and eigenfunctions of $P\mathcal{A}$.
\end{remark}

In the proof of Theorem \ref{thm:gap1} we have shown that, if $G$ is $\Delta$--regular, then $\varepsilon=\frac{1}{\Delta-1}$, but we can also say something more. In fact, since it is known that $\pm 1$ are always eigenvalues of $\mathcal{A}$, it follows that, if $G$ is $\Delta$--regular, then $1\pm \frac{1}{\Delta-1}$ are two real eigenvalues of $\mathcal{L}$. A natural question is whether regular graphs are the only graphs for which $\varepsilon=\frac{1}{\Delta-1}$, but the answer is no. As shown in Theorem \ref{thm:dcycle} in the next section, for instance, also the presence of a $\Delta$--regular cycle in the graph $G$ produces the eigenvalues $1\pm \frac{1}{\Delta-1}$ for $\mathcal{L}$.

\section{Cycles}\label{section:cycles}
We now study the signature that the presence of cycles can leave in the spectrum of the non-backtracking Laplacian. Again, throughout this section we fix a simple graph $G=(V,E)$ with minimum vertex degree $\geq 2$. We let $\mathcal{G}=(\mathcal{V},\mathcal{E})$ be its non-backtracking graph, and we let $\mathcal{L}$ be its non-backtracking Laplacian.\newline

Observe that, given $c_1,\ldots,c_\ell \in V$, we have that (Figure \ref{fig:cycles})
\begin{align*}
   \{c_1,\ldots,c_\ell\} \text{ is a simple chordless cycle in }G \iff & \mathcal{C}\coloneqq\{[c_1, c_2],\ldots,[c_\ell, c_1]\} \\
   &\text{is a simple chordless cycle in }\mathcal{G}\\
     \iff &\text{both }\mathcal{C} \text{ and }\mathcal{C}^{-1}\coloneqq\{[c_1, c_\ell],\ldots,[c_2, c_1]\} \\
    & \text{are simple chordless cycles in }\mathcal{G}.
\end{align*}

\begin{figure}[h]
    \centering
    \includegraphics[width=12cm]{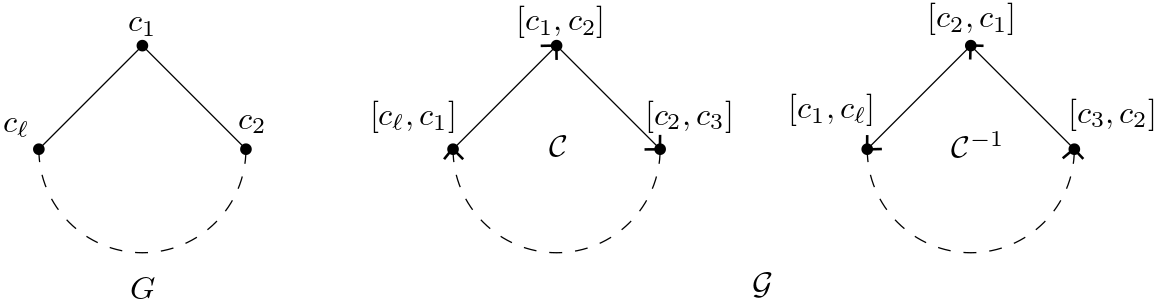}
    \caption{The setting of Section \ref{section:cycles}.}
    \label{fig:cycles}
\end{figure}

Moreover, $\mathcal{C}\cap \mathcal{C}^{-1}=\emptyset$ and there are no edges between $\mathcal{C}$ and $\mathcal{C}^{-1}$ in $\mathcal{G}$. \newline

Throughout the section, we also fix $\mathcal{C}$ and $\mathcal{C}^{-1}$ as above. Our first result concerns a structural property of $\mathcal{G}$ related to the vertices outside $\mathcal{C}\sqcup \mathcal{C}^{-1}$:

\begin{lemma}\label{lemma:outC}
If  $e\in \mathcal{V}\setminus\left(\mathcal{C}\sqcup \mathcal{C}^{-1}\right)$, then exactly one of the following holds:
\begin{enumerate}
    \item There are no edges between $e$ and $\mathcal{C}\sqcup \mathcal{C}^{-1}$ in $\mathcal{G}$;
    \item There exists $i\in \{1,\ldots,\ell\}$ such that the only edge between $e$ and $\mathcal{C}$ is $(e\rightarrow [c_{i}, c_{i+1}])$, while the only edge between $e$ and $\mathcal{C}^{-1}$ is $(e\rightarrow [c_{i}, c_{i-1}])$ (Figure \ref{fig:Lemmacycles});
     \item There exists $i\in \{1,\ldots,\ell\}$ such that the only edge between $e$ and $\mathcal{C}$ is $([c_{i-1}, c_{i}]\rightarrow e)$, while the only edge between $e$ and $\mathcal{C}^{-1}$ is $([c_{i+1}, c_{i}]\rightarrow e)$.
\end{enumerate}
\end{lemma}

\begin{figure}[h]
    \centering
    \includegraphics[width=7cm]{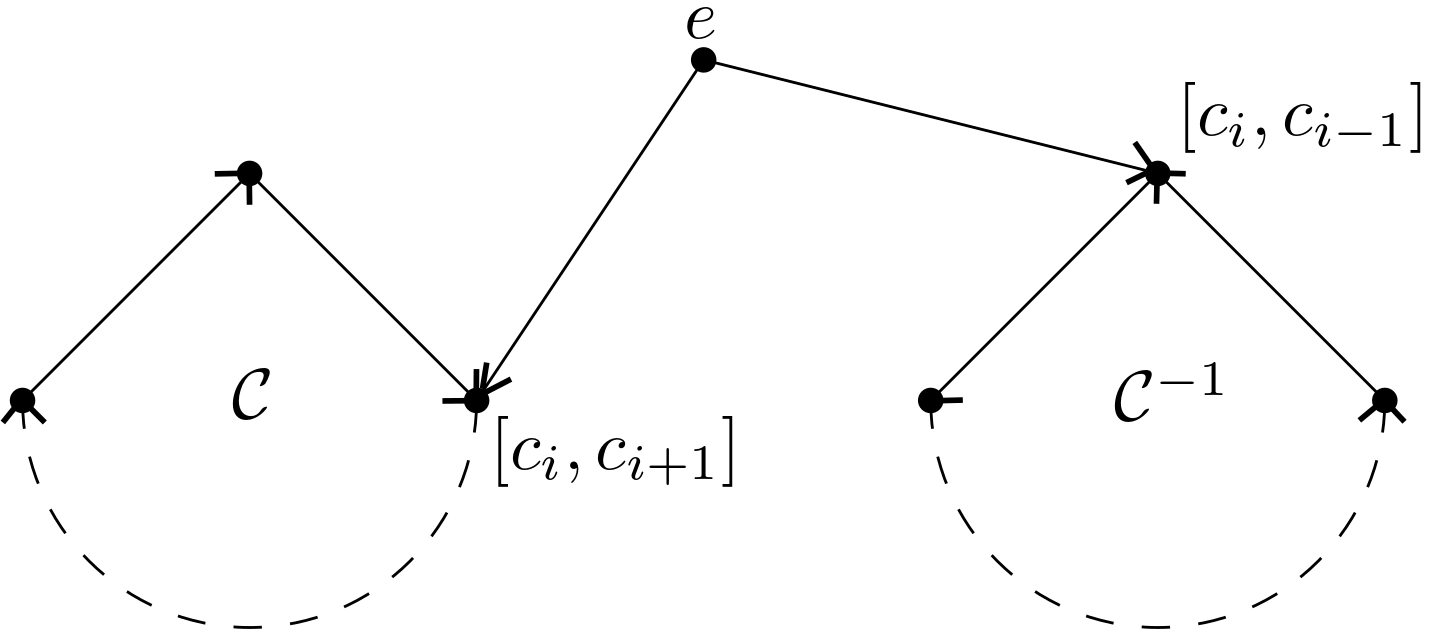}
    \caption{An illustration of Lemma \ref{lemma:outC}.}
    \label{fig:Lemmacycles}
\end{figure}
\begin{proof}
 Assume that there exists a directed edge going from $e$ to $\mathcal{C}\sqcup \mathcal{C}^{-1}$ in $\mathcal{G}$ and assume, without loss of generality, that this directed edge goes to $\mathcal{C}$. Then, there exists $i\in \{1,\ldots,\ell\}$ such that $(e\rightarrow [c_{i}, c_{i+1}])\in \mathcal{E}$. This implies that $\out(e)=c_i$ and $\inp(e)\neq c_{i-1},c_{i+1}$, since we are assuming that $e\notin \mathcal{C}\sqcup \mathcal{C}^{-1}$. Therefore, $(e\rightarrow[c_{i}, c_{i-1}])\in \mathcal{E}$. Moreover, since $\out(e)=c_i$ and $\inp(e)\neq c_{i-1},c_{i+1}$, it is clear that there cannot be other edges between $e$ and $\mathcal{C}\sqcup \mathcal{C}^{-1}$, other than $(e\rightarrow[c_{i}, c_{i+1}])$ and $(e\rightarrow[c_{i}, c_{i-1}])$.\newline 
 In the same way one can show that, if there exists a directed edge going from $\mathcal{C}\sqcup \mathcal{C}^{-1}$ to $e$ in $\mathcal{G}$, then there exists $i\in \{1,\ldots,\ell\}$ such that the only edge between $e$ and $\mathcal{C}$ is $([c_{i-1}, c_{i}]\rightarrow e)$, while the only edge between $e$ and $\mathcal{C}^{-1}$ is $([c_{i+1}, c_{i}]\rightarrow e)$. This proves the claim.
\end{proof}

The next two results will allow us to prove Theorem \ref{thm:dcycle} below, which, as anticipated in the previous section, shows that the presence of a $d$--regular cycle in the graph $G$ produces the eigenvalues $1\pm \frac{1}{d-1}$ for $\mathcal{L}$.

\begin{lemma}\label{lemma:conditions}
Let $\mu\in \mathbb{R}_+$ and let $f:\mathcal{V}\rightarrow \mathbb{R}$, $f\neq \mathbf{0}$ be a function whose support is contained in $\mathcal{C}\sqcup \mathcal{C}^{-1}$. Then, $(1\mp \frac{1}{\mu},f)$ is an eigenpair for $\mathcal{L}$ if and only if the following conditions hold:
\begin{enumerate}
    \item \begin{equation*}
    f([c_{i-1}, c_i])=\pm\frac{\mu}{\deg c_i-1}\cdot f([c_i, c_{i+1}]), \quad \forall i\in\{1,\ldots,\ell\};
\end{equation*}
    \item \begin{equation*}
    f([c_{i+1}, c_i])=\pm\frac{\mu}{\deg c_i-1}\cdot f([c_i, c_{i-1}]), \quad \forall i\in\{1,\ldots,\ell\};
\end{equation*}
\item 
\begin{equation*}
       f([c_i, c_{i+1}])+f([c_i, c_{i-1}])=0, \quad \forall i\in\{1,\ldots,\ell\}\,:\, \deg c_i>2.
   \end{equation*} \end{enumerate}
\end{lemma}
\begin{proof}
 It is clear that $(1\mp\frac{1}{\mu},f)$ is an eigenpair for $\mathcal{L}$ if and only if 
\begin{equation}\label{eq:f-mu}
    f([v, w])=\pm\frac{\mu}{\deg w-1}\cdot \left(\sum_{[w, z], \, z\neq v}f([w, z])\right), \quad \forall [v, w].
\end{equation}Since, by assumption, the support of $f$ is contained in $\mathcal{C}\sqcup \mathcal{C}^{-1}$, it is clear that \eqref{eq:f-mu} applied to the vertices in $\mathcal{C}\sqcup \mathcal{C}^{-1}$ gives the first two conditions of the claim.\newline 
Now, let $e=[v, w]\in \mathcal{V}\setminus\left(\mathcal{C}\sqcup \mathcal{C}^{-1}\right)$. If there are no directed edges from $e$ to $\mathcal{C}\sqcup \mathcal{C}^{-1}$ in $\mathcal{G}$, then \eqref{eq:f-mu} is trivially satisfied. Otherwise, by Lemma \ref{lemma:outC}, there exists $i\in \{1,\ldots,\ell\}$ such that the only edge from $e$ to $\mathcal{C}$ is $(e\rightarrow [c_{i}, c_{i+1}])$, while the only edge from $e$ to $\mathcal{C}^{-1}$ is $(e\rightarrow [c_{i},c_{i-1}])$. This happens if and only if $\out(e)=c_i$ and, in this case, \eqref{eq:f-mu} becomes 
   \begin{equation*}
       f([c_{i}, c_{i+1}])+f([c_{i}, c_{i-1}])=0.
   \end{equation*}
   Since an element $e\in \mathcal{V}\setminus\left(\mathcal{C}\sqcup \mathcal{C}^{-1}\right)$ such that $\out(e)=c_i$ exists if and only if $\deg c_i>2$, this proves the claim.

\end{proof}
 
 \begin{corollary}\label{cor:g}
 Let $\mu\in \mathbb{R}_+$ and let $f:\mathcal{V}\rightarrow \mathbb{R}$, $f\neq \mathbf{0}$ be a function whose support is contained in $\mathcal{C}\sqcup \mathcal{C}^{-1}$. If $(1- \frac{1}{\mu},f)$ is an eigenpair for $\mathcal{L}$ and $\mathcal{C}$ has even length, then also $1+ \frac{1}{\mu}$ is an eigenvalue for $\mathcal{L}$ and it has an eigenfunction whose support is contained in $\mathcal{C}\sqcup \mathcal{C}^{-1}$.
 \end{corollary}
 \begin{proof}
 Let $g:\mathcal{V}\rightarrow \mathbb{R}$ be defined by $g(e)\coloneqq0$ if $e\in \mathcal{V}\setminus\left(\mathcal{C}\sqcup \mathcal{C}^{-1}\right)$ and
 \begin{equation*}
     g([c_i, v])\coloneqq\begin{cases}f([c_i, v]), & \text{ if } i \text{ is even}\\
     -f([c_i, v]), & \text{ if } i \text{ is odd},
     \end{cases}
 \end{equation*}for $i=1,\ldots,\ell.$ Then, $g$ satisfies the conditions in Lemma \ref{lemma:conditions} for $1+ \frac{1}{\mu}$.
 \end{proof}
 
 The previous results allow us to prove the following 
 \begin{theorem}\label{thm:dcycle}
Let $d\in\mathbb{N}_{>1}$. If there exists a simple chordless cycle in $G$ whose vertices have constant degree $d$, then $1-\frac{1}{d-1}$ is an eigenvalue for $\mathcal{L}$. If, additionally, such a cycle is even, then also $1+\frac{1}{d-1}$ is an eigenvalue for $\mathcal{L}$.\newline 
Moreover, the geometric multiplicity of $1-\frac{1}{d-1}$ for $\mathcal{L}$ is larger than or equal to the number of $d$--regular simple chordless cycles in $G$, while the geometric multiplicity of $1+\frac{1}{d-1}$ for $\mathcal{L}$ is larger than or equal to the number of $d$--regular even simple chordless cycles in $G$.
 \end{theorem}

 \begin{proof}
 Given a cycle whose vertices have constant degree $d$, let $f:\mathcal{V}\rightarrow \mathbb{R}$ be defined by
 \begin{equation*}
     f(e)\coloneqq\begin{cases}1 & \text{ if }e\in\mathcal{C}\\
     -1 & \text{ if }e\in\mathcal{C}^{-1}\\
     0 & \text{ otherwise.}\end{cases}
 \end{equation*}
  Then, $f$ satisfies the conditions in Lemma \ref{lemma:conditions} for the eigenvalue $1-\frac{1}{d-1}$, implying that $(1-\frac{1}{d-1},f)$ is an eigenpair for $\mathcal{L}$. This proves the first claim for $1-\frac{1}{d-1}$, and the first claim for $1+\frac{1}{d-1}$ then follows from Corollary \ref{cor:g}.\newline 
  The second claim follows from the fact that the above functions are linearly independent if they are defined for distinct cycles.
 \end{proof}
 
 Similarly, we also prove the following result.
 
  \begin{theorem}\label{thm:d}
Let $d\in\mathbb{N}_{>2}$. If there exists a simple chordless cycle of length $\ell$ in $G$ such that one vertex has degree $d$ while all other vertices have degree $2$, then $1-\frac{1}{\sqrt[\ell]{d-1}}$ is an eigenvalue for $\mathcal{L}$. If, additionally, such cycle is even, then also $1+\frac{1}{\sqrt[\ell]{d-1}}$ is an eigenvalue for $\mathcal{L}$.\newline 
Moreover, the geometric multiplicity of $1-\frac{1}{\sqrt[\ell]{d-1}}$ for $\mathcal{L}$ is larger than or equal to the number of such cycles in $G$, while the geometric multiplicity of $1+\frac{1}{\sqrt[\ell]{d-1}}$ for $\mathcal{L}$ is larger than or equal to the number of such even cycles in $G$.
 \end{theorem}
 
 \begin{proof}Fix a cycle $\{c_1,\ldots,c_\ell\}$ in $G$ and let $\mathcal{C}\coloneqq\{[c_1, c_2],\ldots,[c_\ell, c_1]\}$, $\mathcal{C}^{-1}\coloneqq\{[c_1, c_\ell],\ldots,[c_2, c_1]\}$ be the corresponding cycles in $\mathcal{G}$. Let $c_0 \coloneqq c_\ell$ and $c_{\ell+1} \coloneqq c_1$, and assume that $\deg c_1=d>2$, while $\deg c_i=2$ for $i\in \{2,\ldots,\ell\}$. The proof is similar to the one of Theorem \ref{thm:dcycle}. In this case, we can apply Lemma \ref{lemma:conditions} for $\mu=\sqrt[\ell]{d-1}$ if we find a non-zero function $f:\mathcal{V}\rightarrow \mathbb{R}$ whose support is contained in $\mathcal{C}\sqcup \mathcal{C}^{-1}$ and such that the following conditions hold:
\begin{enumerate}
    \item \begin{equation*}
    f([c_{\ell}, c_1])=\frac{\sqrt[\ell]{d-1}}{d-1}\cdot f([c_1, c_2]);
\end{equation*}
 \item \begin{equation*}
    f([c_{i-1}, c_i])=\sqrt[\ell]{d-1}\cdot f([c_i, c_{i+1}]), \quad \forall i\in\{2,\ldots,\ell\};
\end{equation*}
    \item \begin{equation*}
    f([c_2, c_1])=\frac{\sqrt[\ell]{d-1}}{d-1}\cdot f([c_1, c_\ell]);
\end{equation*}
 \item \begin{equation*}
    f([c_{i+1}, c_i])=\sqrt[\ell]{d-1}\cdot f([c_i, c_{i-1}]), \quad \forall i\in\{2,\ldots,\ell\};
\end{equation*}
\item 
\begin{equation*}
       f([c_1, c_2])+f([c_1, c_\ell])=0.
   \end{equation*} \end{enumerate}
 
 By letting
 \begin{align*}
      f([c_1, c_2])&\coloneqq-f([c_1, c_\ell])\coloneqq1,\\
       f([c_{\ell}, c_1])&\coloneqq-f([c_2, c_1])\coloneqq\frac{\sqrt[\ell]{d-1}}{d-1},\\
       f([c_{\ell-1}, c_\ell])&\coloneqq\sqrt[\ell]{d-1}\cdot f([c_{\ell}, c_1]),\\
       &\,\vdots\\
        f([c_2, c_3])&\coloneqq\sqrt[\ell]{d-1}\cdot f([c_3, c_4]),\\
        f([c_3, c_2])&\coloneqq\sqrt[\ell]{d-1}\cdot f([c_2, c_1]),\\
        &\,\vdots\\
        f([c_\ell, c_{\ell-1}])&\coloneqq\sqrt[\ell]{d-1}\cdot f([c_{\ell-1}, c_{\ell-2}]),
 \end{align*}then clearly conditions 1, 2 and 5 above are satisfied, as well as condition 2 and for $i\in \{3,\ldots,\ell\}$, and condition 4 for $i\in\{2,\ldots,\ell-1\}$. Moreover, since
 \begin{align*}
 \sqrt[\ell]{d-1}\cdot f([c_2, c_3])&=\sqrt[\ell]{d-1}\cdot\sqrt[\ell]{d-1}\cdot f([c_3, c_4])\\ &= (\sqrt[\ell]{d-1})^{\ell-1}\cdot f([c_{\ell}, c_1])\\
 &=(\sqrt[\ell]{d-1})^{\ell-1}\cdot \frac{\sqrt[\ell]{d-1}}{d-1}\\
 &=1=f([c_1, c_2]),
 \end{align*}the second condition is satisfied also for $i=2$. Similarly, since
 \begin{equation*}
    \sqrt[\ell]{d-1}\cdot f([c_\ell, c_{\ell-1}])=(\sqrt[\ell]{d-1})^{\ell-1}\cdot f([c_2, c_1])=-1=f([c_1, c_\ell]),
\end{equation*}the fourth condition is satisfied also for $i=\ell$. By using the same method as in the proof of Theorem \ref{thm:dcycle}, this proves the claim.
  \end{proof}

  The previous results tell us that, for certain families of simple chordless cycles $\mathcal{C}$ in $\mathcal{G}$, the presence of such cycles produces an eigenvalue of the form $1-\frac{1}{\mu}$ for $\mathcal{L}$, where $\mu$ is a positive real number that depends on the cycle structure. Moreover, this eigenvalue admits eigenfunctions whose support is contained in $\mathcal{C}\sqcup\mathcal{C}^{-1}$. A natural question is whether this can be generalized to all cycles. The answer is no, as shown by the following result, which gives a complete characterization of the cycles for which it is possible.
  
  \begin{theorem}
  Assume that $G$ is not a cycle graph. Let $\{c_1,\ldots,c_\ell\}$ be a simple chordless cycle in $G$ and let $$\mathcal{C}\coloneqq\{[c_1, c_2],\ldots,[c_\ell, c_1]\}, \quad \mathcal{C}^{-1}\coloneqq\{[c_1, c_\ell],\ldots,[c_2, c_1]\}$$ be the corresponding cycles in $\mathcal{G}$. Then, $\mathcal{L}$ admits an eigenpair of the form $(1-\frac{1}{\mu},f)$, where $\mu\in \mathbb{R}_+$ and $f:\mathcal{V}\rightarrow \mathbb{R}$ is a non-zero function with support in $\mathcal{C}\sqcup\mathcal{C}^{-1}$, if and only if 
\begin{equation*}
         \mu=\sqrt[\ell]{(\deg c_1-1)\cdot (\deg c_2-1)\cdots (\deg c_\ell-1)}
     \end{equation*}and, up to re-labeling the vertices of the cycle\footnote{After re-labeling, we still want $\{c_1,\ldots,c_\ell\}$ to be a cycle. Hence, for instance, if $c_4$ becomes $c_2$, then $c_5$ can only become either $c_1$ or $c_3$.}, $\deg c_1>2$ and, for all $i\in\{2,\ldots,\ell\}$ such that $\deg c_i>2$,

     \color{black}
      \begin{align}
      \nonumber & \frac{\mu}{\deg c_{i+1}-1}\cdot \frac{\mu}{\deg c_{i+2}-1}\cdots  \frac{\mu}{\deg c_{\ell}-1}\cdot \frac{\mu}{\deg c_{1}-1}\\
      &=\frac{\mu}{\deg c_{i-1}-1}\cdot \frac{\mu}{\deg c_{i-2}-1}\cdots \frac{\mu}{\deg c_2-1}\cdot \frac{\mu}{\deg c_{1}-1}.\label{cond:3b}
   \end{align}
  \end{theorem}
\begin{proof}
Since $G$ is not a cycle graph, then up to re-labeling the vertices of the cycle, we can assume that $\deg c_1>2$. In this case, $(1-\frac{1}{\mu},f)$ is an eigenpair for $\mathcal{L}$ if and only if they satisfy the conditions of Lemma \ref{lemma:conditions}. In particular, it is clear by these conditions that $f([c_1, c_2])\neq 0$, therefore up to normalizing $f$ we can assume that $$f([c_1, c_2])=1.$$ The third condition of Lemma \ref{lemma:conditions} applied to $i=1$ then gives
 \begin{equation*}
     f([c_1, c_\ell])=-1,
 \end{equation*}while the first condition applied to $i\neq 2$ gives
 \begin{align*} 
       f([c_{\ell}, c_1])&=\frac{\mu}{\deg c_1-1}\cdot f([c_1, c_2])=\frac{\mu}{\deg c_1-1},\\
       f([c_{\ell-1}, c_\ell])&=\frac{\mu}{\deg c_\ell-1}\cdot f([c_\ell, c_1])=\frac{\mu}{\deg c_\ell-1}\cdot \frac{\mu}{\deg c_1-1},\\
       &\,\vdots\\
       f([c_2, c_3])&=\frac{\mu}{\deg c_3-1}\cdot f([c_3, c_4])=\frac{\mu}{\deg c_3-1}\cdots\frac{\mu}{\deg c_\ell-1} \cdot  \frac{\mu}{\deg c_1-1},
 \end{align*}
 and the second condition of Lemma \ref{lemma:conditions} applied to $i\neq \ell$ gives
 \begin{align*} 
       f([c_2, c_1])&=\frac{\mu}{\deg c_1-1}\cdot f([c_1, c_\ell])=-\frac{\mu}{\deg c_1-1},\\
       f([c_3, c_2])&=\frac{\mu}{\deg c_2-1}\cdot f([c_2, c_1])=-\frac{\mu}{\deg c_2-1}\cdot \frac{\mu}{\deg c_1-1},\\
       &\,\vdots\\
       f([c_\ell, c_{\ell-1}])&=\frac{\mu}{\deg c_{\ell-1}-1}\cdot f([c_{\ell-1}, c_{\ell-2}])=-\frac{\mu}{\deg c_{\ell-1}-1}\cdot \frac{\mu}{\deg c_{\ell-2}-1}\cdots \frac{\mu}{\deg c_1-1}.
 \end{align*}In particular, the above conditions give a complete description of $f$, given $\mu$. We now check the other conditions of Lemma \ref{lemma:conditions}. We observe that:
 \begin{enumerate}
     \item The first condition of Lemma \ref{lemma:conditions} is satisfied for all $i\in \{1,\ldots,\ell\}$ if and only if
     \begin{align*}
         1&=f([c_1, c_2])=\frac{\mu}{\deg c_2-1}\cdot f([c_2, c_3])\\ &=\frac{\mu}{\deg c_2-1}\cdot\frac{\mu}{\deg c_3-1}\cdots\frac{\mu}{\deg c_\ell-1} \cdot  \frac{\mu}{\deg c_1-1},
     \end{align*}i.e., if and only if 
     \begin{equation*}
         \mu=\sqrt[\ell]{(\deg c_1-1)\cdot (\deg c_2-1)\cdots (\deg c_\ell-1)};
     \end{equation*}
     \item The second condition of Lemma \ref{lemma:conditions} is satisfied for all $i\in \{1,\ldots,\ell\}$ if and only if
     \begin{align*}
         -1&=f([c_1, c_\ell])=\frac{\mu}{\deg c_\ell-1}\cdot f([c_\ell, c_{\ell-1}])\\
         &=-\frac{\mu}{\deg c_\ell-1}\cdot\frac{\mu}{\deg c_{\ell-1}-1}\cdot \frac{\mu}{\deg c_{\ell-2}-1}\cdots \frac{\mu}{\deg c_1-1},
     \end{align*}i.e., as before, if and only if
     \begin{equation*}
         \mu=\sqrt[\ell]{(\deg c_1-1)\cdot (\deg c_2-1)\cdots (\deg c_\ell-1)};
     \end{equation*}
     \item The third condition of Lemma \ref{lemma:conditions} is satisfied if and only if, for all $i\in\{2,\ldots,\ell\}$ such that $\deg c_i>2$,
     \begin{equation}\label{cond:3}
       f([c_{i}, c_{i+1}])=-f([c_{i}, c_{i-1}]),
   \end{equation}where
   \begin{equation*}
       f([c_{i}, c_{i+1}])=\frac{\mu}{\deg c_{i+1}-1}\cdot \frac{\mu}{\deg c_{i+2}-1}\cdots  \frac{\mu}{\deg c_{\ell}-1}\cdot \frac{\mu}{\deg c_{1}-1}
   \end{equation*}and
    \begin{equation*}
       -f([c_i, c_{i-1}])=\frac{\mu}{\deg c_{i-1}-1}\cdot \frac{\mu}{\deg c_{i-2}-1}\cdots \frac{\mu}{\deg c_1-1},
   \end{equation*}therefore \eqref{cond:3} can be rewritten as
   \begin{equation*}
       \frac{\mu}{\deg c_{i+1}-1}\cdot \frac{\mu}{\deg c_{i+2}-1}\cdots  \frac{\mu}{\deg c_{\ell}-1}=\frac{\mu}{\deg c_{i-1}-1}\cdot \frac{\mu}{\deg c_{i-2}-1}\cdots \frac{\mu}{\deg c_2-1}.
   \end{equation*}
 \end{enumerate}
 Putting everything together, we have that $1-\frac{1}{\mu}$ is an eigenvalue with an eigenfunction whose support is contained in $\mathcal{C}\sqcup \mathcal{C}^{-1}$ if and only if \begin{equation*}
         \mu=\sqrt[\ell]{(\deg c_1-1)\cdot (\deg c_2-1)\cdots (\deg c_\ell-1)}
     \end{equation*}and \eqref{cond:3b} is satisfied for all $i\in\{2,\ldots,\ell\}$ such that $\deg c_i>2$. \end{proof}

\section{Cospectrality}\label{section:cospectrality}

We now discuss cospectrality with respect to several matrices associated with the same graph.  As before, for simplicity, given a simple graph $G=(V,E)$, we let $\mathcal{G}=(\mathcal{V},\mathcal{E})$ denote its non-backtracking graph, let $\mathcal{L}$ be its non-backtracking Laplacian, and let $\mathcal{A}$ be the transpose of the non-backtracking matrix of $G$.

Given two undirected graphs, they are said to be \emph{$X$--cospectral}, \emph{$X$-cospectral mates} or simply \emph{$X$-mates} if the eigenvalue spectrum of their respective matrices $X$ is the same, including multiplicities.
For example, two graphs are $A$--cospectral if the eigenvalues of their adjacency matrices are the same, or $\mathcal{L}$--cospectral if the eigenvalues of their non-backtracking Laplacians are the same.
If a graph has no $X$--mate, it is said to be \emph{determined by its $X$--spectrum}.

Cospectrality with respect to the adjacency matrix $A$ has a long history, see \cite{van2009developments} and Chapter 14 of \cite{brouwer} for extensive reviews and for instance \cite{Thune12} for further results. 
In particular, it is known that almost all trees admit an $A$--mate \cite{schwenk1973almost}, meaning that as the number of nodes increases, the fraction of trees that admit an $A$--mate goes to one.
The case of non-trees remains an open question though it is conjectured that the opposite holds: that the fraction of non-tree graphs determined by their $A$--spectrum goes to one as the number of nodes increases \cite{van2009developments}.
Constructions for $A$--cospectral graphs are well known \cite{godsil1982constructing,brouwer}, as are results on graphs determined by their $A$--spectrum and that of their complement \cite{wang2006sufficient}.

Cospectrality with respect to the non-backtracking matrix $\mathcal{A}$ is less well-understood.
First, recall that the $\mathcal{A}$ eigenvalues of a tree are all zero.
In this sense, $\mathcal{A}$ is worse than $A$ at distinguishing trees based on the eigenvalue spectrum, as the fraction of trees with an $\mathcal{A}$--cospectral mate is always equal to one.
However, it is also known \cite{durfee2015distinguishing} that among graphs with minimum degree $\ge 2$ and at most $11$ nodes, the number that admits an $\mathcal{A}$--mate is considerably smaller than the number of such graphs that admit an $A$--mate.\footnote{In Table 5.2 in \cite{durfee2015distinguishing}, the columns labeled $A$, $Z$ correspond to graphs determined by their $A$--spectrum and $\mathcal{A}$--spectrum, respectively.}
This has led to the conjecture that almost all graphs with minimum degree $\ge 2$ are determined by their $\mathcal{A}$--spectrum and, more strongly, that almost all graphs that admit an $A$--mate do not admit an $\mathcal{A}$--mate.
Importantly, recall that when the number of nodes grows large, almost all graphs have minimum degree $\ge 2$.

The strongest evidence for these conjectures comes from exhaustive computations of cospectrality among all graphs with a small number of nodes.
Here we present similar calculations involving the number of graphs not determined by their spectrum of their $A$, $L$, $\mathcal{A}$, and $\mathcal{L}$ matrices.
In this regard our results extend those of \cite{brouwer2009cospectral,durfee2015distinguishing}.
To the best of our knowledge, this is the first time that cospectrality with respect to $L$ has been studied in this exhaustive way, though specific constructions are known \cite{butler2011construction}.
At the time of writing, the results in \cite{brouwer2009cospectral,durfee2015distinguishing}, and the present work provide the most complete picture of the number of cospectral graphs with small number of nodes.
See also relevant entries in the Online Encyclopedia of Integer Sequences \cite{A006611,A006608,A006610,A099881}, and references therein.

\subsection{Computational details}
For the sake of reproducibility, we detail the procedure followed to obtain the computational results.

Among all unlabeled graphs with a number of nodes $4 \le N \le 10$, we find the graphs with a cospectral mate with respect to the matrices $A$, $L$, $\mathcal{A}$, $\mathcal{L}$ via direct computation of their spectra and direct comparison.
The graphs were generated using the software package \cite{nauty}, and their spectrum was computed using standard software routines in the python programming language. Six decimal places of precision were used.

For a graph $G=(V, E)$ with non-backtracking graph $\mathcal{G}=(\mathcal{V}, \mathcal{E})$, we define the diagonal matrix $\tilde{\mathcal{D}}$ via  $\tilde{\mathcal{D}}_{vv} \coloneqq 1 / \deg v$ when $\deg v > 0$, and $\tilde{\mathcal{D}}_{vv} \coloneqq 0$ when $\deg v = 0$, for each $v \in \mathcal{V}$.
Additionally, define $\tilde{\mathcal{L}} \coloneqq \tilde{D} \mathcal{A}$.
This allows us to compute a non-backtracking Laplacian when $G$ has nodes of degree $1$.
Note if the graph has minimum degree $\ge 2$, we have $\tilde{\mathcal{D}} = \mathcal{D}^{-1}$ and $\tilde{\mathcal{L}} = \mathcal{L}$.
Before computing $\mathcal{A}$ or $\mathcal{L}$ of a graph, any nodes of degree zero are removed as they do not effect the spectrum (or indeed the size) of these two matrices

Note that all graphs with $M=0$ edges and any number of nodes can be considered cospectral to each other with respect to the non-backtracking operators $\mathcal{A}$ and $\tilde{\mathcal{L}}$ as in this case both matrices have size $0 \times 0$.
Additionally, graphs with $M=1$ are also $\mathcal{A}$--cospectral and $\tilde{\mathcal{L}}$--cospectral since in this case both matrices are $2 \times 2$ zero matrices. 
Other than these exceptions, there are no cospectral graphs with $N < 4$ with respect to any of the matrices considered here.

\subsection{Results}
Our calculations on the number of $L$--cospectral and $\mathcal{L}$--cospectral graphs with small number of nodes present evidence to the effect that $L$ has nicer cospectrality properties than $A$.
Similarly, our results show that $\mathcal{L}$ has nicer properties than $\mathcal{A}$.
These claims will be made clear shortly.

Before introducing the results, we note an advantage that $\mathcal{L}$ has over $\mathcal{A}$.
The Ihara-Bass determinant formula \cite{angel2015non,bass1992ihara,hashimoto1989zeta} states
\begin{equation*}
\det \left( I - t \mathcal{A} \right) = \left( 1 - t^2 \right)^{M - N} \det \left( I - t A - t^2\left( D - I \right) \right).
\end{equation*}
In other words, the reciprocals of the eigenvalues of $\mathcal{A}$ are the roots of the polynomial in the right hand side.
This formula states that at least $2M - 2N$ of the eigenvalues of $\mathcal{A}$ are equal to $\pm 1$, meaning the ``bulk'' of the spectrum of $\mathcal{A}$, and the most informative part, is comprised of only $2N$ complex numbers.
On the other hand, there is no such formula for $\mathcal{L}$.
In fact, experimentally we have observed that many graphs have $2M$ distinct eigenvalues for $\mathcal{L}$.
Thus we may expect the spectrum of $\mathcal{L}$ to be in general more expressive than that of $\mathcal{A}$.\newline

Table \ref{tab:graphs} shows the number of graphs that are not determined by their spectrum.
The overall best matrix (i.e.\ the one with fewest such graphs) is $L$ starting at $N \ge 8$.
We note that $A$ and $\mathcal{A}$ seem to have the same order of magnitude of number of graphs not determined by their spectrum, as do $L$ and $\tilde{\mathcal{L}}$, and that there is a similar difference when comparing $A$ to $L$ as when comparing $\mathcal{A}$ to $\tilde{\mathcal{L}}$.
Similar computations have been reported \cite{durfee2015distinguishing} that include the Kirchhoff Laplacian and the signless Laplacian, which have not been considered in this work.
We note that the normalized Laplacian $L$ used here admits fewer graphs with cospectral mates than these two other matrices.

\begin{table}[h]
\centering
\begin{tabular}{r|r|rr|rr}
\toprule
$N$ &    \#graphs &      $A$ &    $L$ &  $\mathcal{A}$ &  $\tilde{\mathcal{L}}$ \\
\midrule
$\le$4  &     17 &          0 &       4 &            4 &           4 \\
5  &          34 &          2 &      12 &           11 &           8 \\
6  &         156 &         10 &      32 &           57 &          26 \\
7  &       1 044 &        110 &     108 &          363 &         100 \\
8  &      12 346 &      1 722 &     413 &        3 760 &         574 \\
9  &     274 668 &     51 039 &   1 824 &       64 221 &       4 622 \\
10 &  12 005 168 &  2 560 606 &  26 869 &    1 936 969 &      57 356 \\
\midrule
total & 12 293 427 & 2 613 489 & 29 262 & 2 005 385 & 62 690 \\
\bottomrule
\end{tabular}
\caption{Graphs not determined by their spectrum, by number of nodes $N$.}
\label{tab:graphs}
\end{table}

To elucidate how much the numbers in Table \ref{tab:graphs} are influenced by the existence of nodes of degree $1$,
Table \ref{tab:md2} shows the number of graphs with minimum degree $\ge 2$ that are not determined by their spectrum.
Recall in this case $\tilde{\mathcal{L}} = \mathcal{L}$.
In this data set, the utility of the non-backtracking operators $\mathcal{A}$ and $\tilde{\mathcal{L}}$ is clear, and it points to the fact that the vast majority of graphs in Table \ref{tab:graphs} that are not distinguished by the spectrum of their non-backtracking operators is due to the existence of nodes of degree $1$.
In contrast, the number of graphs not determined by their $A$--spectrum or $L$--spectrum is comparable when considering all graphs versus graphs with minimum degree $\ge 2$.
Furthermore, in Table \ref{tab:md2}, the number of graphs with $\mathcal{L}$--mates is orders of magnitude smaller than all the other matrices studied in this work and other works, and the smallest graphs with an $\mathcal{L}$--mate are larger than the smallest graphs with a $X$--mate, for $X \in \{A, L, \mathcal{A}\}$.\newline

\begin{table}[h]
\centering
\begin{tabular}{r|r|rr|rr}
\toprule
$N$ &  \#graphs &      $A$ &    $L$ &  $\mathcal{A}$ &  $\tilde{\mathcal{L}}$ \\
\midrule
$\le$6  &         76 &          0 &       2 &               0 &              0 \\
7  &        510 &         26 &       4 &               0 &              0 \\
8  &      7 459 &        744 &      11 &               2 &              0 \\
9  &    197 867 &     32 713 &     243 &               6 &              0 \\
10 &  9 808 968 &  1 976 884 &  16 114 &          10 130 &            156 \\
\midrule
total & 10 014 880 & 2 010 367 & 16 374 & 10 138 & 156 \\
\bottomrule
\end{tabular}
\caption{Graphs with minimum degree $\ge 2$ not determined by their spectrum, by number of nodes $N$.}
\label{tab:md2}
\end{table}

Another interesting feature of cospectrality is the size of each cospectrality class.
Given a graph $G$, the size of the $X$--cospectrality class of $G$ is the number of graphs that are $X$--cospectral to it.
The size of a cospectrality class can be arbitrarily large, and large classes usually point to the failure of the spectrum of some matrix as a useful way to distinguish between graphs.
For example, as mentioned earlier, all trees with the same number of edges and any number of nodes are in the same cospectrality class with respect to $\mathcal{A}$ and $\tilde{\mathcal{L}}$.
Table \ref{tab:pairs} shows the percentage of graphs not determined by their spectra whose cospectrality class has size two, i.e.\ the graphs that have exactly one distinct mate, among all graphs and among graphs with minimum degree $\ge 2$.
Interestingly, we see that $A$ and $\mathcal{A}$ have a decreasing tendency in the number of classes of size two, while $L$ and $\mathcal{L}$ show an increasing trend.
Once again, the utility of $\mathcal{L}$ is clear: every single graph with up to $N=10$ nodes and minimum degree $\ge 2$ that has a $\mathcal{L}$--mate has the fewest number of such mates, namely exactly one.\newline

\begin{table}
\centering
\begin{tabular}{r|rrrr|rrrr}
\toprule
{} & \multicolumn{4}{c|}{All} & \multicolumn{4}{c}{Min. deg. $\ge 2$} \\
$N$ &     $A$ &    $L$ & $\mathcal{A}$ & $\tilde{\mathcal{L}}$ &               $A$ &     $L$ & $\mathcal{A}$ & $\tilde{\mathcal{L}}$ \\
\midrule
4  &     --- &   0.00 &    100.00 &    100.00 &       --- &     --- &       --- &    --- \\
5  &  100.00 &  50.00 &     72.73 &     25.00 &       --- &     --- &       --- &    --- \\
6  &  100.00 &  56.25 &     35.09 &     38.46 &       --- &  100.00 &       --- &    --- \\
7  &   94.55 &  48.15 &     16.53 &     46.00 &    100.00 &  100.00 &       --- &    --- \\
8  &   89.55 &  60.53 &      4.15 &     56.10 &     94.62 &   72.73 &    100.00 &    --- \\
9  &   82.39 &  68.64 &      1.62 &     65.82 &     84.17 &   98.77 &    100.00 &    --- \\
10 &   78.37 &  91.48 &      1.31 &     74.99 &     79.33 &   99.59 &     99.88 & 100.00 \\
\bottomrule
\end{tabular}
\caption{Percentage of graphs not determined by their spectrum whose cospectrality class has size $2$, by number of nodes $N$.}
\label{tab:pairs}
\end{table}

In the case of the non-backtracking operators $\mathcal{A}$ and $\tilde{\mathcal{L}}$, Tables \ref{tab:graphs}--\ref{tab:pairs} show the number of graphs with the \emph{same number of nodes and edges} that are $\mathcal{A}$--cospectral or $\tilde{\mathcal{L}}$--cospectral.
This is due to the fact that the size of $\mathcal{A}$ and $\tilde{\mathcal{L}}$ depends only on the number of edges.
It is also possible that there exist graphs with the same number of edges but different number of nodes that are cospectral with respect to these non-backtracking operators. Note \cite{durfee2015distinguishing} argue that if the graph has minimum degree at least $2$, then the spectrum determines both the number of nodes and edges.
Nevertheless, Table \ref{tab:byedge} shows such instances.
This table shows that among graphs with $4 \le N \leq 10$ and minimum degree $\ge 2$, no graphs with different number of nodes are $\mathcal{A}$--cospectral or $\mathcal{L}$--cospectral (compare the totals between Tables \ref{tab:graphs}--\ref{tab:byedge}).
However, there are many graphs with nodes of degree less than $2$ that are cospectral to some other graph with a different number of nodes.
This should not be a surprise; for instance, consider any graph $G$ and define $G'$ as the disjoint union of $G$ with the singleton graph.
Then, $G$ and $G'$ are both $\mathcal{A}$--cospectral and  $\tilde{\mathcal{L}}$--cospectral (the respective matrices from each graph are in fact the same matrix).
Whether such trivial cases are a complete explanation of the numbers seen in Table \ref{tab:byedge} remains an open question.
Finally, we point out that the number of $\mathcal{L}$--cospectral graphs as a function of the number of edges $M$ is in progression: $4, 8, 16, 24, 26, 26, 14, 16, 8, 4$; see right-most column of Table \ref{tab:byedge}.
We hypothesize this is not an accident but part of a larger pattern, though further research is needed to establish it fully.

\begin{table}
\centering
\small
\begin{tabular}{r|r|rr|r|rr}
\toprule
{} & \multicolumn{3}{c|}{All} & \multicolumn{3}{c}{Min. deg. $\ge 2$} \\
$M$ &    \#graphs &       $\mathcal{A}$ &      $\tilde{\mathcal{L}}$ &             \#graphs &     $\mathcal{A}$ &   $\tilde{\mathcal{L}}$ \\
\midrule
0  &          7 &        7 &       7 &           0 &     0 &   0 \\
1  &          7 &        7 &       7 &           0 &     0 &   0 \\
2  &         14 &       14 &      14 &           0 &     0 &   0 \\
3  &         32 &       32 &      32 &           0 &     0 &   0 \\
4  &         60 &       60 &      60 &           1 &     0 &   0 \\
5  &        118 &      118 &     118 &           2 &     0 &   0 \\
6  &        254 &      253 &     254 &           6 &     0 &   0 \\
7  &        521 &      517 &     521 &          10 &     0 &   0 \\
8  &      1 117 &    1 079 &   1 117 &          25 &     0 &   0 \\
9  &      2 429 &    2 246 &   2 424 &          68 &     0 &   0 \\
\midrule
10 &      5 233 &    4 633 &   5 157 &         182 &     0 &   0 \\
11 &     11 148 &    9 930 &  10 483 &         532 &     0 &   0 \\
12 &     23 215 &   20 744 &  18 860 &       1 679 &     0 &   0 \\
13 &     46 439 &   40 831 &  29 681 &       5 218 &     4 &   0 \\
14 &     88 645 &   74 294 &  41 833 &      15 437 &    14 &   0 \\
15 &    159 965 &  123 304 &  53 790 &      41 126 &    26 &   0 \\
16 &    270 897 &  184 297 &  63 814 &      96 274 &    62 &   0 \\
17 &    428 559 &  246 821 &  70 080 &     197 433 &   162 &   0 \\
18 &    630 899 &  295 705 &  71 155 &     355 986 &   364 &   4 \\
19 &    861 535 &  317 166 &  66 368 &     567 827 &   634 &   8 \\
\midrule
20 &  1 089 368 &  305 084 &  56 680 &     807 284 &   983 &  16 \\
21 &  1 273 438 &  263 655 &  44 169 &   1 029 639 & 1 329 &  24 \\
22 &  1 374 523 &  205 247 &  31 362 &   1 184 688 & 1 492 &  26 \\
23 &  1 368 996 &  144 209 &  20 305 &   1 235 599 & 1 490 &  26 \\
24 &  1 257 395 &   91 728 &  12 029 &   1 172 658 & 1 333 &  24 \\
25 &  1 064 416 &   52 858 &   6 525 &   1 015 663 &   989 &  16 \\
26 &    830 367 &   27 717 &   3 284 &     804 863 &   628 &   8 \\
27 &    596 963 &   13 319 &   1 547 &     584 762 &   368 &   4 \\
28 &    395 512 &    5 877 &     691 &     390 136 &   166 &   0 \\
29 &    241 725 &    2 415 &     296 &     239 514 &    60 &   0 \\
\midrule
30 &    136 496 &      948 &     126 &     135 636 &    26 &   0 \\
31 &     71 343 &      351 &      50 &      71 025 &     8 &   0 \\
32 &     34 674 &      126 &      22 &      34 559 &     0 &   0 \\
33 &     15 777 &       48 &      10 &      15 734 &     0 &   0 \\
34 &      6 761 &       18 &       4 &       6 745 &     0 &   0 \\
35 &      2 770 &        7 &       2 &       2 764 &     0 &   0 \\
36 &      1 104 &        4 &       2 &       1 101 &     0 &   0 \\
$\ge$37 &        705 &        0 &       0 &         704 &     0 &   0 \\
\midrule
total & 12 293 427 & 2 435 669 & 612 879 & 10 014 880 & 10 138 & 156 \\
\bottomrule
\end{tabular}
\caption{\small{Graphs with $4 \le N \le 10$ with an $\mathcal{A}$--mate or $\mathcal{L}$--mate, by number of edges $M$, among all graphs and among graphs with minimum degree $\ge 2$.}}
\label{tab:byedge}
\end{table}

Let us now consider the $78$ pairs of $\mathcal{L}$--mates with $N=10$; Figures \ref{fig:mates1} and \ref{fig:mates2}  show the smallest such pairs.
Our calculations confirm that each of these $78$ pairs of $\mathcal{L}$--mates is also a pair of $A$--mates, $L$--mates, and $\mathcal{A}$--mates, i.e.\ they cannot be distinguished using the spectrum of any of the matrices used in this work.
In other words, our computations show that the most effective way to distinguish between two graphs with up to $N \le 10$ nodes and minimum degree $\ge 2$ using cospectrality methods is by using our non-backtracking Laplacian $\mathcal{L}$.\newline

Based on the results shown here, and in parallel to the open conjectures regarding graphs being determined by their $A$--spectrum, we state the following.

\begin{conjecture}
Almost all graphs with minimum degree $\ge 2$ are determined by the spectrum of their non-backtracking Laplacian $\mathcal{L}$.
\end{conjecture}

\begin{conjecture}
Almost all graphs with minimum degree $\ge 2$ with an $\mathcal{A}$--mate have no $\mathcal{L}$--mate.
\end{conjecture}

\begin{figure}[t]
    \centering
    \includegraphics[width=9cm]{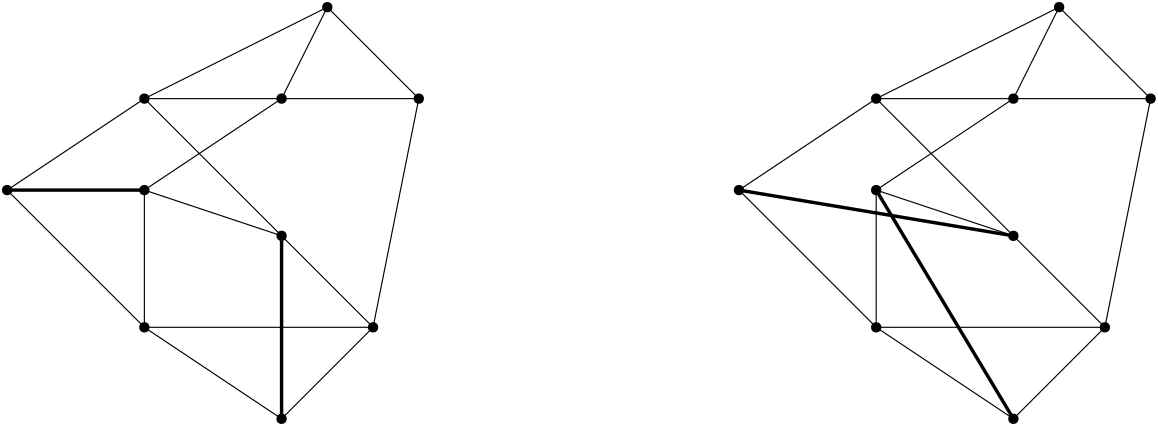}
    \caption{One of the two smallest pairs of $\mathcal{L}$-mates.}
    \label{fig:mates1}
\end{figure}

\begin{figure}[t]
    \centering
    \includegraphics[width=9cm]{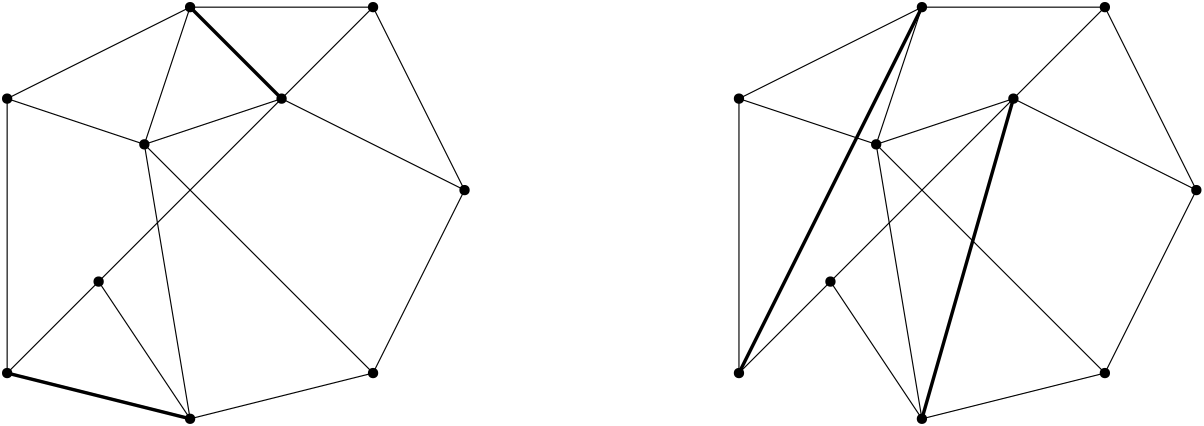}
    \caption{One of the two smallest pairs of $\mathcal{L}$-mates.}
    \label{fig:mates2}
\end{figure}

\section*{Acknowledgments}
Raffaella Mulas was supported by the Max Planck Society's Minerva Grant.

\bibliographystyle{plain} 

\bibliography{NB11-10-22}

\end{document}